\documentclass[11pt, letter]{article}

\usepackage[margin=1.1in]{geometry} 
\usepackage{parskip}
\usepackage{longtable,color,natbib}
\usepackage{booktabs} 
\usepackage{array} 
\usepackage{paralist} 
\usepackage{verbatim} 
\usepackage{appendix}
\usepackage{amssymb}
\usepackage{amsthm}
\usepackage{algorithm,algorithmic}
\usepackage{amsmath}
\usepackage{bbm}
\usepackage{enumitem}
\usepackage{mdframed}
\usepackage{adjustbox}
\usepackage{mathrsfs}

\definecolor{linkcolor}{rgb}{0.1,0,0.7}
\definecolor{urlcolor}{rgb}{1,0,0}
\usepackage[unicode=true,colorlinks,citecolor=linkcolor,linkcolor=linkcolor,urlcolor=blue]{hyperref}
\usepackage{graphicx} 
\usepackage{epstopdf}
\usepackage{bm}
\usepackage{cleveref}
\usepackage{pst-node}
\psset{arrowscale=2,arrows=->}

\usepackage{wrapfig}
\usepackage{multirow}
\usepackage{tabularx}
\usepackage[all,cmtip]{xy}
\usepackage{bm}

\usepackage{subcaption}
\usepackage{graphicx}

\usepackage{tikz} 
\usetikzlibrary{positioning, arrows.meta}
\usetikzlibrary{calc}

\tikzset{
  frame/.style={
    rectangle, draw,
    text width=6em, text centered,
    minimum height=4em,drop shadow,fill=white,
    rounded corners,
  },
  line/.style={
    draw, -latex',rounded corners=3mm,
  }
}

\newtheorem{Theorem}{Theorem}[section]

\newtheorem{Proposition}[Theorem]{Proposition}
\newtheorem{Lemma}[Theorem]{Lemma}
\newtheorem{Corollary}[Theorem]{Corollary}
\newtheorem{Remark}[Theorem]{Remark}

\newtheorem{Assumption}[Theorem]{Assumption}

\numberwithin{equation}{section}
\usepackage{hhline}
\usepackage{verbatim}
\usepackage{eurosym}

\newcommand{\nc}{\newcommand}
\nc{\ind}{\mathds{1}}
\def \trans{^{\scriptscriptstyle{\intercal}}}

\newcommand{\R}{\mathbb{R}}

\newcommand{\E}{\mathcal{E}}
\newcommand{\F}{\mathcal{F}}

\newcommand{\G}{\mathcal{G}}

 \DeclareMathOperator{\Exp}{Exp}

\allowdisplaybreaks[4]

\DeclareMathOperator{\esssup}{esssup}

\def\esssup_#1{\underset{#1}{\mathrm{ess\,sup\, }}}
\def\essinf_#1{\underset{#1}{\mathrm{ess\,inf\, }}}
\def\argmax_#1{\underset{#1}{\mathrm{arg\,max\, }}}
\def\argmin_#1{\underset{#1}{\mathrm{arg\,min\, }}}

\def\b1{\bf 1}

\def \N{\mathbb{N}}
\def \R{\mathbb{R}}

\def \E{\mathbb{E}}
\def \F{\mathbb{F}}
\def \G{\mathbb{G}}

\def \P{\mathbb{P}}

\def \Ac{{\cal A}}

\def \Dc{{\cal D}}
\def \Ec{{\cal E}}
\def \Fc{{\cal F}}
\def \Gc{{\cal G}}
\def \Hc{{\cal H}}

\def \Lc{{\cal L}}
\def \Pc{{\cal P}}

 \def \Nc{{\cal N}}
 
\def \Sc{{\cal S}}
\def \Tc{{\cal T}}
\def \Uc{{\cal U}}

\def \Wc{{\cal W}}

\def\eqref#1{{\rm(\ref{#1})}}
\def\beqs{\begin{eqnarray*}}
\def\enqs{\end{eqnarray*}}
\def\beq{\begin{eqnarray}}
\def\enq{\end{eqnarray}}

\begin{document}

\title{Continuous-time q-learning for mean-field control with\\ common noise, part-II: q-learning algorithms}

\author{Zhenjie Ren \thanks{Email: zhenjie.ren@univ-evry.fr, LaMME, Universit\'e \'Evry Paris-Saclay, \'Evry, France.}
\and
Xiaoli Wei \thanks{Email: tyswxl@gmail.com.}
\and
Xiang Yu \thanks{Email: xiang.yu@polyu.edu.hk, Department of Applied Mathematics, The Hong Kong Polytechnic University, Kowloon, Hong Kong.}
\and
Xun Yu Zhou \thanks{Email: xz2574@columbia.edu, Department of Industrial Engineering and Operations Research, Columbia University, New York, USA.}
}
\date{This version: April 30, 2026}

\maketitle
\begin{abstract}
This paper is a continuation work of \cite{Renetal25} aiming to further devise q-learning algorithms for mean-field control (MFC) with controlled common noise. Based on the relaxed control formulation, we first establish the martingale condition of the value function and the Iq-function by evaluating along the conditional state distributions generated by all test policies. As the data in the relaxed control formulation are not observable in practice, we quantify the error incurred when they are replaced by the observable ones in the exploratory formulation under discretely sampled actions. This, together with a two-layer fixed point characterization of an optimal policy in \cite{Renetal25}, allows us to propose several algorithms including the Actor-Critic q-learning algorithm, in which the policy is updated in the Actor-step based on the iteration rule induced by the improved Iq-function, and the value function and Iq-function are updated in the Critic-step based on the martingale orthogonality condition using the data from the exploratory formulation. We also establish the convergence of the inner iterations in the Actor-step in an infinite-horizon linear quadratic (LQ) framework. In two examples, within and beyond LQ framework, our q-learning algorithms are implemented with satisfactory performance.
\end{abstract}
%
\section{Introduction}
Thanks to its tractability and scalability, MFC has become a popular paradigm to approximate the large system of cooperative agents who are coordinated by the social planner to achieve the social optimum. In most conventional methods for solving MFC problems, one needs to assume the full knowledge or precise estimations of the underlying model. However, in real-life applications, the agent or the social planner may only have limited or no information about the environment, which may cause huge errors or inefficiency in implementing the theoretical solutions. This motivated fast-growing developments in reinforcement learning (RL) approaches for MFC problems. Based on the principle of trial-and-error in the unknown environment, the social planner can gradually learn to select best actions from the exploration and exploitation procedure. See some early studies on discrete-time mean-field Markov decision processes and RL algorithms among \cite{GGWX21, GGWX22, AFL2022, AFHR2023, CLT23, Cui23}, as well as \cite{Lauriereetal22} for an overview of discrete-time RL approaches in mean-field models. 

However, many real-world stochastic systems, particularly in finance and engineering, evolve continuously through time, which call for suitable RL algorithms in a continuous-time framework. Recently, for classical single agent's stochastic control problems, \cite{Wangetal2021}, \cite{JZ22a, JZ22b, jiazhou2022} have laid the theoretical foundation for RL with entropy regularization in continuous time models with continuous state space and action space. In particular, \cite{Wangetal2021} discussed the optimal policy in an entropy-regularized exploratory formulation for continuous-time RL, where the entropy term plays the crucial role to encourage the exploration. \cite{JZ22a} examined the policy evaluation problem by establishing a martingale condition of the value function. \cite{JZ22b} investigated the policy gradient algorithm by connecting it to the martingale condition and policy evaluation in \cite{JZ22a}. \cite{jiazhou2022} pioneered the continuous-time q-learning theory by introducing the q-function as the first order time derivative of the advantage function and establishing a joint martingale characterization of the q-function and the value function. The continuous-time
entropy-regularized RL approach has been rapidly generalized in different single agent's settings and applications, see, for instance,  \cite{W23,Han23,Dong2022,DDJ,BHY23,DDJZ2023,giegrichetal2024,HLYZ,J26}. Recently, the continuous-time exploratory formulation and some RL algorithms have also been developed in mean-field models without common noise, see among \cite{GuoXZ,FGLPS23,Liangetal2024,weiyu2023,weiyuyuan2023}.

In \cite{Renetal25}, we proposed a proper definition of Iq-function and contributed some new theoretical foundations for the exploratory MFC problem with controlled common noise. The social planner therein is responsible for the learning task on behalf of the population, who assigns randomized policies to agents and lets them interact with the environment based on their own current states and population's state distribution. Based on agents' interactions, the social planner collects the population's distribution and agent's individual rewards to generate the aggregated reward and improve the policy through iterations. Some applications that fit into this framework could be the centralized traffic management system that coordinates the vehicles on the road and the central bank that sets economic policy such as tax rates to maximize a social welfare function.  

 In previous studies on single agent's control problems and MFC without common noise, the optimal policy resulting from the exploratory HJB equation is typically characterized as a one-layer fixed point of the Gibbs measure. In sharp contrast, the controlled common noise in \cite{Renetal25} creates an additional nonlinear functional of policy, which significantly complicates the definition of Iq-function and its connection to the optimal policy. As a main finding of \cite{Renetal25}, it is revealed that the optimal policy corresponds to a two-layer fixed point of the argmax operator of the Iq-function, which poses some interesting open problems in devising the model-free RL algorithms: How to learn the Iq-function and how to learn the optimal policy via the implicit policy iteration and the learnt Iq-function?

As a continuation of our work \cite{Renetal25}, the present paper is devoted to addressing these new challenges in q-learning algorithms based on whether the two-layer fixed point admits an explicit expression or not. We summarize the main contributions of this paper as three-fold:
\begin{itemize}
\item [(i)] First, based on the relaxed control formulation, we establish the martingale conditions of the Iq-function and the value function for a given policy and the optimal policy, respectively. Note that these martingale conditions rely on the data generated in the relaxed control formulation, including the conditional state distribution and the reward trajectories, which are not directly observable in practice. In the implementation step, we need to utilize observable data generated from the exploratory formulation under discretely sampled actions. Therefore, it becomes important to discuss whether the theoretical martingale condition can still be employed in the policy evaluation. In fact, unlike the result in \cite{jiazhou2022} in the single agent's setting that the martingale conditions remain valid with data from the sampled SDE, the counterpart in our mean-field model may generally fail. In response, we establish a new error analysis of the martingale condition in our MFC problem when we substitute the data by the one generated by discretely sampled actions (see Theorem \ref{thm:grid-approximation} and Proposition \ref{prop:grid-approximation}).
\item [(ii)] When the optimal policy as a two-layer fixed point can be obtained explicitly such as in the LQ setting in \cite{Renetal25}, we devise an optimal q-learning algorithm (see Figure \ref{qlearnfig} and Algorithm \ref{algo:fixed-point}) by taking the advantage of the fact that the optimal Iq-function and the optimal policy share the same parameters in the algorithm. On the other hand, for the general case when the two-layer fixed point may not admit an explicit expression, we devise an Actor-Critic q-learning algorithm to allow different parameterizations of the Iq-function and the policy. In the Actor step, we update the policy on strength of the implicit policy improvement rule via the improved Iq-function, leading to the operator for parameter iterations involving the partial linear functional derivative of the unregularized Iq-function (see \eqref{estimate-gradient}). In the Critic step, we update the parameters in Iq-function and the value function simultaneously based on the averaged martingale orthogonal condition (see Figure \ref{ACqfig} and Algorithm \ref{algo:actor-critic}). We also discuss the subtle issue in the Actor step to implement the implicit policy improvement iterations, i.e. the inner iterations of parameters to find the optimal one-step iterated policy. The convergence of the inner fixed point iterations is rigorously established in the LQ framework over an infinite horizon.
\item [(iii)] To illustrate the efficiency of our proposed algorithms, we study two examples within and beyond the LQ framework. In both examples, we are able to derive the explicit form of the optimal policy and the resulting optimal value function and the optimal Iq-function, allowing us to obtain the exact parameterizations of these target functions. We then implement and compare three algorithms, namely the optimal q-learning algorithm, the Actor-Critic q-learning algorithm without inner fixed point iterations, and the Actor-Critic q-learning algorithm with inner fixed point iterations. In numerical illustrations, all three algorithms show satisfactory performance, and it is interesting to observe in two examples that the inner fixed point iterations to learn the optimal one-step iterated policy may not be necessary as the alternative iterations in Actor and Critic steps will achieve the same goal to converge to the two-layer fixed point.
\end{itemize}

The remainder of this paper is organized as follows. Section \ref{sec:background} reviews the formulations and theoretical findings of Iq-function in continuous-time MFC problems with common noise. Section \ref{sec:martingale}  establishes the martingale conditions for the Iq-function and the value function. Section \ref{sec:alg} presents two q-learning algorithms, namely the optimal q-learning algorithm when the two-layer fixed point characterization of the optimal policy admits an explicit form and the Actor-Critic q-learning algorithm otherwise by resorting to the policy improvement result. Section \ref{actor-convergence-LQ} verifies the convergence of the inner iterations for the LQ MFC problems over an infinite horizon. Section \ref{sec:examples} presents two numerical examples using our proposed q-learning algorithms.

\paragraph{Notations} Given two Polish spaces $(S, \Sc)$ and $(T, \Tc)$, for any $p >0$, we denote by $\Pc_p(S)$ the space of all probability measures with finite $p$-th moments on $S$ and equip $\Pc_p(S)$ with the $p$-Wasserstein metric $\Wc_p$ defined by
\begin{align*}
\Wc_p(\mu, \nu) = \inf \Big\{\Big(\int_{S \times S} |x-y|^p \gamma(dx, dy)\Big)^{1/p}: \gamma \in \Pc_p(S \times S)\; \mbox {has marginals} \; \mu \; \mbox{and} \; \nu\Big\}.
\end{align*}
For $\mu \in \Pc_p(\R^d)$, denote $\|\mu\|_p: = \big(\int_{\R^d} |x|^p \mu(dx)\big)^{1/p}$. We denote by $\Pc_{ac}(S)$ the space of probability measures on $S$ that are absolutely continuous with respect to the Lebesgue measure. We denote by $\Pc_{ac}(T|S)$ the space of all probability kernels from $S$ to $\Pc_{ac}(T)$. We denote $L^2(\Omega, \Fc, \P; \R^d)$ the space of all $\Fc$-adapted $\R^d$-valued square-integrable random variables on the probability space $(\Omega, \Fc, \P)$. For any measurable function $g:S\to \R^k$, we denote $\int_{S} g(x)\mu(dx)$ as $\langle g, \mu\rangle$. For a functional $F: \Pc_2(S) \to \R$, we denote by $\partial_\mu F(\mu)(x)$, $\partial_x\partial_\mu F(\mu)(x)$ and $\partial_\mu^2 F(\mu)(x, x')$ the $L$ derivative in $\mu$, the mixed second-order derivative with respect to $\mu$ and $x$, and the second-order derivative  in measure $\mu$, respectively. We adopt $\Nc(\mu, \Sigma)$ and $\Uc([p, q])$ for a Gaussian distribution with mean $\mu$ and covariance $\Sigma$ and a uniform distribution on $[p, q]$, respectively.

\section{Setup and Preparations}\label{sec:background}

\subsection{Exploratory MFC with common noise in RL}

We work with an unknown mean-field model, where the social planner is responsible for the RL task who needs to learn the optimal policy based on the trial-and-error procedure. To begin with, let us first introduce two RL formulations of MFC with common noise as discussed in  \cite{Renetal25}: a relaxed control formulation that is needed for theoretical analysis and an exploratory formulation under discretely sampled actions that is suitable for the implementation of learning. 

Let $(\Omega^e, \Fc^e, \P^e)$ be a complete probability space with a product structure $(\Omega^0 \times \Omega^1 \times \Omega^2, \Fc^0 \otimes \Fc^1 \otimes \Fc^2, \P^0 \otimes \P^1 \otimes \P^2)$, where  $(\Omega^0, \Fc^0, \P^0)$ supports a $n$-dimensional Brownian motion $B = (B_s)_{s \in [0, T]}$ that serves as common noise,  $(\Omega^1, \Fc^1, \P^1)$ supports an $m$-dimensional Brownian motion $W=(W_s)_{s \in [0, T]}$, and $(\Omega^2, \Fc^2, \P^2)$ supports a sequence of i.i.d. uniform random variables $(U_i)_{i \in \mathbb{N}}$ for the action randomization. We denote by $\F^{B, W} = (\Fc_s^{B, W})_{0 \leq s \leq T}$ the filtration generated by $B, W$ and by $\G = (\Gc_s)_{0 \leq s \leq T}$ the filtration generated by $B$, respectively. It is also assumed that there exists a sub-algebra $\Hc$ of $\Fc^1$ for the initial state such that $\Hc$ is independent of $\Fc^1$ and it is ``rich enough" in the sense that for any $\mu \in \Pc_2(\R^d)$, there exists an $\Hc$-measurable random variable $\xi$ on $(\Omega^1, \P^1)$ such that $\P^1 \circ \xi^{-1} = \mu$. We denote by $\Fc_s^e = \Fc_s^{B, W} \vee \Hc \vee \sigma(U_i, s_i \leq s)$. 

Throughout the paper, we shall make the following assumption.
\begin{Assumption} \label{ass:model}The following conditions for the dynamics and reward functions hold true.
\begin{enumerate}[label=\upshape(\roman*)]
\item $b$, $\sigma$, $\sigma_o$, $r:$  $[0, T] \times \R^d \times \Pc_2(\R^d) \times \Ac$ $\to$ $\R^d, \R^{d \times m}$, $\R^{d \times n}$, $\R$ are jointly continuous, and $g: \R^d \times \Pc_2(\R^d) \to \R$ is jointly continuous;
\item $b$, $\sigma$, $\sigma_o$ are uniformly Lipschitz continuous in $x$ and $\mu$, that is, for $f \in \{b, \sigma, \sigma_o\}$, there exists a constant $C >0$ such that
    $|f(t, x, \mu) -f(t, x', \mu')| \leq C\big(|x - x'| + \Wc_2(\mu, \mu')\big)$;
\item There exists some constant $C > 0$ such that for any $(t, x, \mu, a) \in [0, T] \times \R^d \times \Pc_2(\R^d) \times \Ac$,
\begin{align*}
|b(t, x, \mu, a)|  & \leq C\Big(1 + |x| + \|\mu\|_2 + |a|\Big),\\
\big|(\sigma\sigma\trans + \sigma_o\sigma_o\trans)(t, x, \mu, a)\big| &\leq  C\big(1 + |x|^2 + \|\mu\|_2^2 + |a|^2\big),\\
|r(t, x, \mu, a)| & \leq C\big(1 + |x|^2 + \|\mu\|_2^2 + |a|^2\big).
\end{align*}
\end{enumerate}
\end{Assumption}

For theoretical analysis, we first consider the relaxed control formulation. Let $\Pi$ stand for the set of admissible policies (see Definition 2.2 in \cite{Renetal25}). For ${\bm \pi} \in \Pi$, we denote by
\begin{align*}
{f}_{\bm \pi}(t, x, \mu) &:= \int_{\Ac} f(t, x, \mu, a) {\bm \pi}(a|t, x, \mu) da, \;\;\; \mbox{for}\; f \; \in \{b, \sigma, \sigma_o\},\\
{\rm std}_{\bm \pi}(f)  &:= \Big(\int_{\Ac} ff\trans (t, x, \mu, a){\bm \pi}(a|t, x, \mu)da - f_{{\bm \pi}} f_{{\bm \pi}}\trans(t, x, \mu)\Big)^{1/2}, \;\;\mbox{for}\; f \; \in \{b, \sigma, \sigma_o\},\\
E_{\bm \pi}(t, x, \mu): &= - \int_{\Ac} \log {\bm \pi}(a|t, x, \mu) {\bm \pi}(a|t, x, \mu)da,\quad \mathcal{E}(t,\mu, {\bm \pi}) :=\int_{\R^d} E_{\bm \pi}(t, x, \mu)\mu(dx), \\
 \hat r_{\bm \pi}(t, \mu): &= \int_{\R^d} r_{\bm \pi}(t, x, \mu)\mu(dx), \quad \hat g(t, \mu) : = \int_{\R^d} g(x, \mu)\mu(dx).
\end{align*}
The state dynamics in the relaxed control formulation is given by
\begin{align}
d X_s^{\bm \pi} &= b_{\bm \pi}(s, X_s^{\bm \pi}, \mu_s^{\bm \pi})ds +  \sigma_{\bm \pi}(s, X_s^{\bm \pi}, \mu_s^{\bm \pi})dW_s + \sigma_{o, {\bm \pi}}(s,  X_s^{\bm \pi}, \mu_s^{\bm \pi})dB_s \nonumber\\
 &\;\;\; + {\rm std}_{{\bm \pi}}(\sigma)(s, X_s^{\bm \pi}, \mu_s^{\bm \pi})d{\overline W}_s + {\rm std}_{\bm \pi}(\sigma_o)(s, X_s^{\bm \pi}, \mu_s^{\bm \pi})d {\bar B}_s, \; s \in [t, T], \label{equ:common-noise-average}
\end{align}
where $X_t^{\bm \pi} = \xi \sim \mu$, $\mu_s^{\bm \pi}$ is the conditional law of $X_s^{\bm \pi}$ given $\Gc_s$, and $\overline W$ and $\bar B$ are {\it extra} one-dimensional Brownian motions independent of $W$ and $B$.
The optimal value function of the MFC problem is $\tilde J^*(t, \mu) = \sup_{{\bm \pi} \in \Pi} \tilde J(t, \mu; {\bm \pi})$, where
\begin{align}\label{equ:exploratory_average_value_function}
\tilde J(t, \mu; {\bm \pi}) 
& = \E^e\Big[\int_t^T e^{-\beta(s -t)}\Big(\hat r_{\bm \pi} (s, \mu_s^{\bm \pi}) + \gamma \mathcal{E}(s, \mu_s^{\bm \pi}, {\bm \pi})\Big)ds + e^{-\beta (T - t)}\hat g(\mu_T^{\bm \pi})\Big].
\end{align}

The relaxed control formulation, while suitable for theoretical analysis, cannot be directly observed in the learning procedure. Instead, we consider the exploratory formulation for the implementation. Moreover, to avoid measure-theoretical issues caused by continuous sampling, we adopt the exploratory formulation under discretely sampled actions. That is, for a given time gird $\Dc = \{t=s_0  < \ldots < s_n = T\}$ with $|\Dc| = \max_{0 \leq i \leq n-1} |s_{i+1} - s_i|$, the representative agent takes action randomization only at time grids. Her state process is then given by
\begin{align}\label{equ:exploratory_SDE1}
dX_s^{\Dc, \bm \pi} &= b(s, X_s^{\Dc, \bm \pi}, \mu_s^{\Dc, \bm \pi},  a_{\delta(s)}^{\Dc, \bm \pi}) ds + \sigma(s, X_s^{\Dc, \bm \pi}, \mu_s^{\Dc, \bm \pi}, a_{\delta(s)}^{\Dc, \bm \pi}) dW_s \\
& \;\;\;\;\;+\sigma_o(s, X_s^{\Dc, \bm \pi}, \mu_s^{\Dc, \bm \pi}, a_{\delta(s)}^{\Dc, \bm \pi}) dB_s, \; s \in [t, T],\; X_t^{\Dc, \bm \pi} = \xi \sim \mu,\nonumber
\end{align}
where $\mu_s^{\Dc, \bm \pi}= \Lc(X_s^{\Dc, \bm \pi}|\Gc_s)$, and for $s \in [s_i, s_{i+1})$, $\delta(s) = s_i$ and $a_{s_i} = \phi_{\bm \pi}(s_i, X_{s_i}^{\Dc, \bm \pi}, \mu_{s_i}^{\Dc, \bm \pi}, U_{s_i}) \sim {\bm \pi}(\cdot|s_i, X_{s_i}^{\Dc, \bm \pi}, \mu_{s_i}^{\Dc, \bm \pi})$.
The social planner's goal is to maximize the objective function as $|\Dc| \to 0$, i.e., $J^*(t, \xi) = \sup_{{\bm \pi} \in \Pi} \lim_{|\Dc| \to 0} J^{\Dc}(t, \xi; {\bm \pi})$,
where $J^{\Dc}(t, \xi; {\bm \pi})$ is given by
\begin{align*}
J^{\Dc}(t, \xi; {\bm \pi})& = \E^e\biggl[\int_t^T e^{-\beta(s-t)} \big(r(s, X_s^{\Dc, \bm \pi}, \mu_s^{\Dc, \bm \pi}, a_{\delta(s)}^{\Dc, \bm \pi})
+ \gamma E_{\bm \pi}(\delta(s), X_{\delta(s)}^{\bm \pi}, \mu_{\delta(s)}^{\Dc, \bm \pi}) \big)ds\\
&\;\;\;\;\;\; + e^{-\beta (T -t)}g( X_T^{\Dc, \bm \pi}, \mu_T^{\Dc, \bm \pi})\Big| X_t^{\Dc, \bm \pi}= \xi \biggl].\nonumber
\end{align*}
It is shown in \cite{Renetal25} that, under suitable model assumptions, two value functions in the relaxed control and the exploratory formulations coincide, i.e., $\tilde J(t, \mu; {\bm \pi}) = J(t, \xi; {\bm \pi})$. We therefore shall use $J$ by slight abuse of notation in the rest of the paper.

\subsection{Continuous-time integrated q-function}\label{sec:q-function}
As the foundation of continuous-time q-learning for MFC with common noise, \cite{Renetal25} proposes the definition of the integrated q-function (Iq-function). Specifically, given a  policy ${\bm \pi} \in \Pi$, for any $(t, \mu, {\bm h}) \in [0, T] \times \Pc_2(\R^d) \times \Pc_{ac}(\Ac|\R^d)$, the Iq-function is defined by
\begin{align}\label{def:coupled-q-function}
q^{\gamma}(t, \mu, {\bm h}; {\bm \pi}) =  \frac{\partial J}{\partial t}(t, \mu; {\bm \pi}) - \beta J(t, \mu; {\bm \pi})  + \mathscr{H}(t, \mu, {\bm h}; {\bm \pi}) + \gamma \mathcal{E}(t, \mu, {\bm h}),
\end{align}
where $\mathscr{H}: [0, T] \times \Pc_2(\R^d) \times \Pc_{ac}(\Ac|\R^d) \to \R$ denotes the integrated Hamiltonian
\begin{align*}
\mathscr{H}(t, \mu, {\bm h}; {\bm \pi}) &:= \int_{\R^d \times \Ac} H\big(t, x, \mu, a,\partial_\mu J(t, \mu; {\bm \pi})(x), \partial_x\partial_\mu J(t, \mu; {\bm \pi})(x)\big) {\bm h}(a| x) da \mu(dx)\\
&\;\;\; + \frac{1}{2}\int_{\R^d \times \R^d}{\rm Tr}\Big(\sigma_{o, {\bm h}}(t, x, \mu)\sigma_{o, {\bm h}}(t, x', \mu) \trans \partial_\mu^2 J(t, \mu; {\bm \pi})(x, x')\Big)\mu(dx) \otimes \mu(dx').
\end{align*}
We call $q^0(t, \mu, {\bm h}; {\bm \pi}) = q^\gamma(t, \mu, {\bm h}; {\bm \pi}) - \gamma \mathcal{E}(t, \mu, {\bm h})$ the unregularized Iq-function.
Building upon \cite{Renetal25}, we summarize some key results of Iq-function (Lemma 3.4, Corollary 4.2 and Theorem 3.10 in \cite{Renetal25}) in the following proposition.
\begin{Proposition}\label{thm:policy_improvement}
\begin{itemize}
\item[(i)] The dynamic programming equation of the value function $J$ in \eqref{equ:exploratory_average_value_function} can be written in terms of Iq-function
\begin{align}\label{equ:dynamic-programming-equation}
q^0(t, \mu, {\bm \pi}; {\bm \pi}) + \gamma \mathcal{E}(t, \mu, {\bm \pi}) =0.
\end{align}
\item [(ii)] ${\bm h}^*$ is the unique maximizer of $\max_{{\bm h}\in \Pc_{ac}(\Ac|\R^d)} q^\gamma(t, \mu, {\bm h}; {\bm \pi})$ if and only if ${\bm h}^*$ is the fixed-point of the map $\Phi_{\bm \pi}: \Pi \to \Pi$
 \begin{align}
{\bm h}^*(a|t, x, \mu) = \Phi_{\bm \pi}({\bm h}^*)(t, x, \mu)= \frac{\exp\Big\{\frac{1}{\gamma} \frac{\delta q^0}{\delta {\bm h}}(t, \mu, {\bm h}^*; {\bm \pi})(x, a)\Big\}}{\int_{\Ac} \exp\Big\{\frac{1}{\gamma} \frac{\delta q^0}{\delta {\bm h}}(t, \mu, {\bm h}^*; {\bm \pi})(x, a)\Big\}da},\; t \in [0, T],
\end{align}
where $\frac{\delta q^0}{\delta {\bm h}}(s, \mu, {\bm h}; {\bm \pi})(x, a)$ is the partial linear functional derivative  of $q^0$ with respect to ${\bm h}$, see Definition 3.7 in \cite{Renetal25}.

\item [(iii)] For a given ${\bm \pi} \in \Pi$, if we select a new policy ${\bm \pi}'$ such that either $q^\gamma(s, \mu, {\bm \pi}'; {\bm \pi}) \geq q^\gamma(s, \mu, {\bm \pi}; {\bm \pi})$ for any $s \in [t, T]$ or $\E^e[\int_t^T e^{-\beta(s-t)} q^\gamma(s, \mu_s^{{\bm \pi}'}, {\bm \pi}'; {\bm \pi})ds] \geq \E^e[\int_t^T e^{-\beta(s-t)} q^\gamma(s, \mu_s^{{\bm \pi}'}, {\bm \pi}; {\bm \pi})ds]$ holds, we then have $J(t, \mu; {\bm \pi}') \geq J(t, \mu; {\bm \pi})$. In particular, if ${\bm h}^*$ is the fixed point of the map $\Phi_{\bm \pi}$, then $q^\gamma(t, \mu, {\bm h}^*; {\bm \pi}) \geq q^\gamma(t, \mu, {\bm \pi}; {\bm \pi})$ for any $(t, \mu) \in [0, T] \times \Pc_2(\R^d)$.

\item [(iv)] An optimal policy ${\bm \pi}^*$ satisfies the two-layer fixed point condition
\begin{align}\label{twofix}
{\bm \pi}^*(a|t, x, \mu) = \frac{\exp\Big\{\frac{1}{\gamma} \frac{\delta q^{0, *}}{\delta {\bm h}}(t, \mu, {\bm \pi}^*)(x, a)\Big\}}{\int_{\Ac} \exp\Big\{\frac{1}{\gamma} \frac{\delta q^{0, *}}{\delta {\bm h}}(t, \mu, {\bm \pi}^*)(x, a)\Big\}da}.
\end{align}

\end{itemize}

\end{Proposition}


\section{Martingale Characterization}\label{sec:martingale}
In this section, we provide the martingale characterization of the value function and the Iq-function in a joint manner in the mean-field setting with common noise. The proofs are similar to those for MFC without common noise (see Theorems 4.3 and 4.4 in \cite{weiyu2023}) and hence are omitted.


\begin{Theorem}[Characterization of value function and Iq-function] \label{thm:unknown-q-function}
Let a continuous function $\hat J: [0, T] \times \Pc_2(\R^d) \to \R$ and a continuous function $\hat q^0: [0, T] \times \Pc_2(\R^d) \times \Pc_{ac}(\Ac|\R^d) \to \R$ be given.
Let ${\bm \pi} \in \Pi$ be a root of the following equation of ${\bm h}$
\begin{align}\label{Iq-consistency-equation}
\hat q^0(t, \mu, {\bm h}) + \gamma\Ec(t, \mu, {\bm h}) =0
\end{align}
for each fixed $(t, \mu) \in [0, T] \times \Pc_2(\R^d)$.
Then $\hat J$ and $\hat q^0$ are  the value function and the unregularized Iq-function associated with ${\bm \pi}$ being the root of \eqref{Iq-consistency-equation} if and only if
$\hat J$ and $\hat q^0$ satisfy
$\hat J(T, \mu) = \hat g(\mu)$,
and for any $(t, \mu, {\bm h}) \in [0, T] \times \Pc_2(\R^d) \times \Pi$, the process
\begin{align}\label{equ:weak_martingale_characterization}
&M_s^{\bm h} : = e^{-\beta s} \hat J(s, \mu_s^{{\bm h}}) + \int_{t}^s e^{-\beta u}\Big(\hat r_{\bm h}(u, \mu_{u}^{{\bm h}}) - \hat q^0 (u, \mu_{u}^{{\bm h}}, {\bm h})\Big)du
\end{align}
is a  $(\P^e, \G)$-martingale, where $\mu_s^{{\bm h}} = \Lc(X_s^{\bm h}|\Gc_s)$ is the conditional state distribution in \eqref{equ:common-noise-average}.
\end{Theorem}

We also provide the martingale characterization of the optimal value function and the optimal Iq-function, which will be used to design the optimal q-learning algorithm.

\begin{Theorem}[Characterization of the optimal value function and optimal Iq-function] \label{thm:cha-optimal-value-q} Let a policy $\hat {\bm \pi}: [0, T] \times \R^d \times \Pc_2(\R^d) \to \Pc(\Ac)$,  a continuous function $\hat J: [0, T] \times \Pc_2(\R^d) \to \R$ and a continuous function $\hat q^0: [0, T] \times \Pc_2(\R^d) \times \Pc_{ac}(\Ac|\R^d) \to \R$ be given. Then $\hat {\bm \pi}$, $\hat J$ and $\hat q^0$ are respectively an optimal policy, the optimal value function and the unregularized optimal Iq-function if and only if the following three conditions hold
\begin{enumerate}[label={{\upshape(\roman*)}}]
\item (Consistency condition) $\hat J$ satisfies
$\hat J(T, \mu) = \hat g(\mu)$,
\item (Two-layer fixed point condition)
 $\hat {\bm \pi}$ satisfies
\begin{align}\label{equ:martingale-characterization-policy}
\left\{
\begin{array}{ll}
\hat q^0(t, \mu, \hat{\bm \pi}) + \gamma \Ec(t, \mu, \hat {\bm \pi}) = 0,\\
\hat {\bm \pi}(a|t, x, \mu) = \frac{\exp\big\{\frac{1}{\gamma}\frac{\delta \hat q^0}{\delta {\bm h}}(t, \mu, \hat {\bm \pi})(x, a)\big\}}{\int_{\Ac} \exp\big\{\frac{1}{\gamma}\frac{\delta \hat q^0}{\delta {\bm h}}(t,  \mu, \hat {\bm \pi})(x, a)\big\}da}.
\end{array}
\right.
\end{align}
\item (Martingale condition) for any $(t, \mu, {\bm h}) \in [0, T] \times \Pc_2(\R^d) \times \Pi$, the process $\{M_s^{\bm h}\}$ in \eqref{equ:weak_martingale_characterization} is a  $(\P^e, \G)$-martingale.
\end{enumerate}
\end{Theorem}

\begin{Remark}\label{rmk:martingale-condition}
In Theorem \ref{thm:unknown-q-function}, the martingale characterization \eqref{equ:weak_martingale_characterization} utilizes the unobservable measure flow $\{\mu_s^{{\bm h}}\}_s$, generated by all test policies ${\bm h}$ in the relaxed control formulation \eqref{equ:common-noise-average}. If the unobservable measure flow $\{\mu_s^{\bm h}\}_s$ is replaced by the observable $\{\mu_s^{\Dc, \bm \pi}\}_s$ in the exploratory formulation \eqref{equ:exploratory_SDE1} under discretely sampled actions, the arguments to prove Theorem \ref{thm:unknown-q-function} become invalid and the process \eqref{equ:weak_martingale_characterization} may no longer be a martingale in the mean-field model. It is therefore crucial to quantify the error of the martingale condition using $\{\mu_s^{\Dc, \bm \pi}\}_s$ in place of $\{\mu_s^{\bm h}\}_s$ in \eqref{equ:weak_martingale_characterization} if we plan to employ this martingale condition in policy evaluation.
\end{Remark}

Let $M_s^{\Dc, \bm h}$ denote the counterpart process of $M_s^{\bm h}$ in \eqref{equ:weak_martingale_characterization} by replacing $\{\mu_s^{\bm h}\}_s$ with $\{\mu_s^{\Dc, \bm \pi}\}_s$, we have the following result on error analysis.

\begin{Theorem}\label{thm:grid-approximation} For any $\G$-adapted process $\{\eta_s^{\bm h}\}_{s \in [t, T]}$ satisfying $\sup_{\bm h \in \Pi}\E^e[\int_0^T |\eta_s^{\bm h}|^2d s]< + \infty$, we have
$\lim_{|\Dc| \to 0}\Big|\E^e\Big[\int_t^T \eta_s^{\bm h} d(M_s^{\bm h} - M_s^{\Dc, {\bm h}})\Big]\Big| =0$.
Consequently, when $\{M_s^{\bm h}\}_s$ is a $(\P^e, \G)$-martingale, $\lim_{|\Dc| \to 0} \E^e\Big[\int_t^T \eta_s^{\bm h} dM_s^{\Dc, \bm h}\Big]=0$.
\end{Theorem}
\begin{proof} Let $\check q^0(t, \mu, {\bm h})$ denote the function obtained from $q^0$ in \eqref{def:coupled-q-function} by replacing $J(t, \mu; {\bm \pi})$ with $\hat J(t, \mu)$. Also denote
\begin{align*}
\Delta q^0(t, \mu, {\bm h}): = & \check q^0(t, \mu, {\bm h}) - \hat q^0(t, \mu, {\bm h}),\\
\Lc \hat J(t, x, x', \mu, a, a'): =& \frac{\partial \hat J}{\partial t}(t, \mu) -\beta \hat J (t, \mu) + b(t, x, \mu, a)\trans \partial_\mu \hat J(t, \mu)(x)\\
& + \frac{1}{2} {\rm Tr}\Big(\big(\sigma\sigma\trans + \sigma_o\sigma_o\trans\big)(t, x, \mu, a) \partial_x\partial_\mu \hat J(t, \mu)(x)\Big)\\
& + \frac{1}{2} {\rm Tr}\big(\sigma_o(t, x, \mu, a) \sigma_o(t, x', \mu, a') \partial_\mu^2 \hat J(t, \mu)(x, x')\big).
\end{align*}
From the expressions of $M_s^{\bm h}$ and $M_s^{\Dc, \bm h}$, we obtain that
\begin{align*}
d(M_s^{\bm h} - M_s^{\Dc, {\bm h}}) =& e^{-\beta s} d\big(\hat J(s, \mu_s^{\bm h}) - \hat J(s, \mu_s^{\Dc, \bm h})\big) -\beta e^{-\beta s} (\hat J(s, \mu_s^{\bm h}) - \hat J(s, \mu_s^{\Dc, \bm h}))ds \\
& + e^{-\beta s} \Big(\hat r_{\bm h}(s, \mu_s^{\bm h}) - \hat r_{\bm h}(s, \mu_s^{\Dc, \bm h}) - \hat q^0(s, \mu_s^{\bm h}, {\bm h}) + \hat q^0(s, \mu_s^{\Dc, \bm h}, {\bm h})\Big)
\end{align*}
By applying It\^o's formula to $d(\hat J(s, \mu_s^{\bm h}) - \hat J(s, \mu_s^{\Dc, {\bm h}}))$, we can rewrite $d(M_s^{\bm h} - M_s^{\Dc, \bm h })$ by
\begin{align*}
d (M_s^{\bm h}- M_s^{\Dc, \bm h}) =& e^{-\beta s}\Big(\Delta q^0(s, \mu_s^{\bm h}, {\bm h})-\E^e \bar \E^e \Big[\Lc \hat J(s, X_s^{\Dc, \bm h}, \bar X_s^{\Dc, \bm h}, \mu_s^{\Dc, \bm h}, a_{\delta(s)}, \bar a_{\delta(s)}) + \hat r_{\bm h}(s, \mu_{s}^{\Dc, {\bm h}}) \\
& - \hat q^0(s, \mu_s^{\Dc, \bm h}, {\bm h})\big|\Gc_s\Big]\Big)ds + e^{-\beta s}\Big(\E^e\big[\partial_\mu \hat J(s, \mu_s^{\bm h})(X_s^{\bm h}) \sigma_{o, {\bm h}}(s,  X_s^{\bm h}, \mu_s^{\bm h})\big|\Gc_s\big]\\
&-\E^e\big[\partial_\mu \hat J(s, \mu_s^{\Dc, \bm h})(X_s^{\Dc, \bm h}) \sigma_o(s, X_s^{\Dc, \bm h}, \mu_s^{\Dc, \bm h}, a_{\delta(s)})\big|\Gc_s\big]\Big)dB_s.
\end{align*}
Noting that $a_{\delta(s)} \sim {\bm h}(\cdot|\delta(s), X_{\delta(s)}^{\Dc, {\bm h}}, \mu_{\delta(s)}^{\Dc, {\bm h}})$ and using the tower property, we have
\begin{align*}
d( M_s^{\bm h}- M_s^{ \Dc, \bm h}) = & e^{-\beta s} \Big( \Delta q^0(s, \mu_s^{\bm h}, {\bm h}) - \Delta  q^0(s, \mu_s^{\Dc, \bm h}, {\bm h})- (\check q^0(\delta(s), \mu_{\delta(s)}^{\Dc, \bm h}, {\bm h}) - \check q^0(s, \mu_s^{\Dc, \bm h}, {\bm h}))\\
& -(\hat r_{\bm h}(s, \mu_s^{\Dc, \bm h})- \hat r_{\bm h}(\delta(s), \mu_{\delta(s)}^{\Dc, {\bm h}}))- \E^e \bar \E^e \Big[\Lc \hat J(s, X_s^{\Dc, \bm h}, \bar X_s^{\Dc, \bm h}, \mu_s^{\Dc, \bm h}, a_{\delta(s)}, \bar a_{\delta(s)}) \\
& - \Lc \hat J(\delta(s), X_{\delta(s)}^{\Dc, \bm h}, \bar X_{\delta(s)}^{\Dc, \bm h}, \mu_{\delta(s)}^{\Dc, \bm h}, a_{\delta(s)}, \bar a_{\delta(s)}) \big|\Gc_s\Big]\Big)ds\\
 & + e^{-\beta s}\Big(\E^e\big[\partial_\mu \hat J(s, \mu_s^{\bm h})(X_s^{\bm h}) \sigma_{o, {\bm h}}(s,  X_s^{\bm h}, \mu_s^{\bm h})\big|\Gc_s\big]\\
 &-\E^e\big[\partial_\mu \hat J(s, \mu_s^{\Dc, \bm h})(X_s^{\Dc, \bm h}) \sigma_o(s, X_s^{\Dc, \bm h}, \mu_s^{\Dc, \bm h}, a_{\delta(s)})|\Gc_s\big]\Big)dB_s\\
 =: &  e^{-\beta s} \big(I_s^{\bm h} ds + II_s^{\bm h} d B_s\big).
\end{align*}
It suffices to estimate $I_s^{\bm h}$ and $II_s^{\bm h}$. By Assumptions \ref{ass:model} in \cite{Renetal25} and standard estimates on $X_s^{\Dc, \bm h}$, it holds that
\begin{align*}
&\E^e\big[|X_s^{\Dc, \bm h} - X_{s_i}^{\Dc, \bm h}|^2\big] + \E^e\big[\Wc_2^2(\mu_{s_i}^{\Dc, \bm h}, \mu_s^{\Dc, \bm h})\big] \leq C |s - s_i| \leq C |\Dc|, s \in [s_i, s_{i+1}],\\
&\E^e[\sup_{s \in [t, T]}|X_s^{\Dc, \bm h}|^2] \leq C.
\end{align*}
By Lemma 3.3 in \cite{Renetal25} with $f (\cdot)= \Delta q^0(\cdot, {\bm h})$, we derive that for $s \in [s_i, s_{i+1}]$
$\E^e[|I_s^{\bm h}|^2]\leq  C |\Dc|$.
In addition, by Assumption \ref{ass:model}, it holds that
\begin{align*}
\E^e\Big[\int_t^T |II_s^{\bm h}|^2 ds \Big] \leq C\E^e \Big[\int_t^T \Big(1 + |X_s^{\Dc, \bm h}|^2 +  \|\mu_s^{\Dc, \bm h}\|_2^2  + |X_s^{\bm h}|^2 + \|\mu_s^{\bm h}\|_2^2\Big) ds \Big] < + \infty,
\end{align*}
and hence $\E^e[\int_0^T \eta_s^{\bm h} II_s^{\bm h} d B_s] = 0$.
It then follows that
\begin{align*}
\Big|\E^e\Big[\int_t^T \eta_s^{\bm h} d(M_s^{\bm h} - M_s^{\Dc, {\bm h}})\Big]\Big|
\leq & \Big(\E^e\Big[\int_t^T |\eta_s^{\bm h}|^2ds\Big]\Big)^{1/2} \Big(\sum_{i=0}^{n-1}\int_{s_i}^{s_{i+1}} \E^e\big[|I_s^{\bm h}|^2\big]ds\Big)^{1/2}\\
\leq & \sup_{\bm h \in \Pi}\Big(\E^e\Big[\int_t^T |\eta_s^{\bm h}|^2ds\Big]\Big)^{1/2} C |\Dc|^{1/2} T^{1/2}.
\end{align*}
We thus conclude that $\lim_{\Dc \to 0} \Big|\E^e\Big[\int_t^T \eta_s^{\bm h} d(M_s^{\bm h} - M_s^{\Dc, {\bm h}})\Big]\Big| =0$.
\end{proof}

\section{Continuous-Time q-Learning Algorithms}\label{sec:alg}

This section contributes to the design of q-learning algorithms based on the theoretical results in previous sections. We first devise and present a learning algorithm to learn the optimal value function and the optimal Iq-function when the optimal policy as a two-layer fixed point \eqref{equ:martingale-characterization-policy} can be derived in an explicit form. As a result, the optimal policy shares the same parameters as that of the optimal Iq-function. In general, however, it could be difficult to explicitly solve the two-layer fixed point  \eqref{equ:martingale-characterization-policy}. Therefore, we need to parameterize the optimal value function, the optimal Iq-function and the optimal policy by distinct parameters respectively, which motivates us to devise the Actor-Critic q-learning algorithm that {\it alternately} updates the Critic (the value function and the Iq-function) and the Actor (the policy). The main differences between the optimal q-learning and the Actor-Critic q-learning are illustrated in Figures \ref{qlearnfig} and
\ref{ACqfig}.

\begin{figure}[htpb]
\centering

\begin{minipage}[t]{0.48\textwidth}
\centering
\raisebox{0cm}{
\resizebox{\linewidth}{!}{%
\begin{tikzpicture}[font=\tiny\sffamily\bfseries, node distance=2cm, >=Stealth]

\node (value) [draw, rounded corners, text width=2.5cm, align=center, minimum height=1cm, fill=orange!10] {
    Value function: $J^{\theta}$ \\
    Iq-function: $q^{0,\psi}$ \\
    Policy: $\pi^{\psi}$
};
\node (environment) [draw, rounded corners, text width=1.5cm, align=center, right=3cm of value, minimum height=1cm, fill=orange!10] {Environment};

\draw[<-, blue] (value.east) -- node[midway, above, align=center] {(2) Samples \\ $(\mu^{m}_{t_k}, \hat{r}^{m}_{t_k}, \mu^{m}_{t_{k+1}})_k$} (environment.west);
\draw[<-, blue] (environment.north) to[out=120,in=60] node[midway, above] {(1) Test policy ${\bm h}^{\tilde{\phi}^m}$} (value.north east);
\draw[<-, blue] (value.south west) to[out=210,in=330,looseness=2] node[below, sloped, midway] {(3) Averaged martingale orthogonal condition} (value.south east);

\end{tikzpicture}}}
\subcaption{Procedure of optimal q-learning algorithm}\label{qlearnfig}
\end{minipage}
\hfill
\begin{minipage}[t]{0.48\textwidth}
\centering
\resizebox{\linewidth}{!}{%
\begin{tikzpicture}[font=\tiny\sffamily\bfseries, node distance=2.5cm, >=Stealth]

\node (policy) [draw, rounded corners, text width=2.5cm, align=center, minimum height=1cm, fill=orange!10] {Policy ${\bm \pi}^{\phi}$};
\node (value) [draw, rounded corners, text width=2.5cm, align=center, minimum height=1cm, fill=orange!10, below=of policy] {Value function: $J^{\theta}$ \\ Iq-function: $q^{0,\psi}$};
\node (environment) [draw, rounded corners, text width=1.5cm, align=center, minimum height=1cm, fill=orange!10, right=of policy, yshift=-1.75cm] {Environment};

\coordinate (midpoint) at ($(policy)!0.5!(value)$);

\draw[->, blue] (policy) edge[bend left=15] node[above, sloped, midway] {(1) Test policy ${\bm h}^{\tilde\phi^m}$} (environment);
\draw[->, blue] (environment.west) -- node[midway, above]{{(2) Samples $(\mu_{t_k}^m, \hat{r}_{t_k}^m, \mu_{t_{k+1}}^m)_k$}} (midpoint);
\draw[->, blue] (value) -- ++(0,1.5) -| node[midway, left] {} (policy);
\draw[->, blue] (policy) -- ++(0,-1.5) -| node[near start, right] {} (value);
\path (value.west) edge [->, blue, >=latex, bend left=50] node[left, yshift=1cm, xshift=-0.25cm, rotate=90]{(4) Policy improvement} (policy.west);
\draw[->, blue] (value.south west) to[out=210,in=330,looseness=2] node[below] {(3) Averaged martingale orthogonal condition} (value.south east);

\node[above=0.15cm of policy] {{\small Actor}};
\node[below=0.15cm of value] {{\small Critic}};

\end{tikzpicture}}
\subcaption{Procedure of Actor-Critic q-learning algorithm}\label{ACqfig}
\end{minipage}

\end{figure}

\subsection{Optimal q-learning algorithm}
In this section, we apply Theorem \ref{thm:cha-optimal-value-q} to learn directly the optimal value function $J^*$ and the optimal Iq-function $q^{0, *}$. Let $J^\theta$, $q^{0, \psi}$ and ${\bm \pi}^\psi$, $(\theta, \psi) \in \Theta \times \Psi$, be respectively the parameterized approximators of $J^*$, $q^{0, *}$ and ${\bm \pi}^*$ such that the consistency condition $J^\theta(T, \mu) = \hat g(\mu)$ and the two-layer fixed point condition \eqref{equ:martingale-characterization-policy} hold at the same time. This happens when we take advantage of the model structure to obtain the explicit form of parameterized function family.

 The rest is to ensure the martingale condition in \eqref{equ:weak_martingale_characterization}. In contrast to \cite{JZ22b}, the feature of mean-field interactions requires us to explore all test policies ${\bm h}$ to meet the martingale condition. Therefore, the martingale orthogonal condition should hold in the average sense, see section 5 in \cite{weiyuyuan2023}. It is impossible to use all test policies in the practical implementation. In response, we take a family of parameterized test policies $\Pi^{\tilde\Phi}=\big\{{\bm h}^{\tilde\phi}: \tilde\phi \in \tilde\Phi\big\}$ for some bounded set $\tilde\Phi \subset \R^{L_{\tilde\phi}}$.
 Denote
\begin{align*}
M_t^{{\bm h}^{\tilde\phi}, \theta, \psi} :=& e^{-\beta t} J^\theta(t, \mu_t^{{\bm h}^\phi}) + \int_{0}^t e^{-\beta u}\big[\hat r_{{\bm h}^{\tilde\phi}}(u, \mu_{u}^{{\bm h}^{\tilde\phi}}) - q^{0, \psi} (u, \mu_{u}^{{\bm h}^{\tilde\phi}}, {\bm h}^{\tilde\phi})\big]du,\\
M_t^{\Dc, {\bm h}^\phi, \theta, \psi} :=& e^{-\beta t} J^\theta(t, \mu_t^{\Dc, {\bm h}^{\tilde\phi}}) + \int_{0}^t e^{-\beta u}\big[\hat r_{{\bm h}^\phi}(u, \mu_{u}^{\Dc, {\bm h}^{\tilde\phi}}) - q^{0, \psi} (u, \mu_{u}^{\Dc, {\bm h}^{\tilde\phi}}, {\bm h}^{\tilde\phi})\big]du.
\end{align*}
According to Theorem \ref{thm:cha-optimal-value-q}, if there exists a pair $(\theta^*, \psi^*) \in \Theta \times \Psi$ such that  $\{M_t^{{\bm h}^{\tilde\phi}, \theta, \psi}\}_t$ is a $(\P^e, \G)$-martingale for any ${\bm h}^{\tilde\phi} \in \Pi^{\tilde\Phi}$, then $J^{\theta^*}$ and $q^{0, \psi^*}$ are the optimal value function and the optimal unregularized Iq-function restricted to $\Pi^{\tilde\Phi}$. A necessary and sufficient condition for the parametrized process $\{M_t^{{\bm h}^{\tilde\phi}, \theta^*, \psi^*}\}_t$ to be a $(\P^e, \G)$-martingale is the following averaged martingale orthogonality condition
\begin{align}\label{AMOC}
0 = \int_{\tilde\Phi} \Big(\E^e\Big[\int_0^T \eta_t^{\tilde\phi} d M_t^{{\bm h}^{\tilde\phi}, \theta^*, \psi^*}\Big]\Big) \Uc(d\tilde\phi),
\end{align}
where  $\eta^{\tilde\phi}$ is any $\G$-progressively measurable test process, and $\Uc$ stands for the uniform distribution on the space $\tilde\Phi$. Recall that $\{M_t^{{\bm h}^{\tilde\phi}, \theta, \psi}\}_t$ is not accessible to the social planner and therefore cannot be used for the learning the parameters $(\theta, \psi)$. We thus obtain the following convergence result when the learning of $(\theta, \psi)$ is based on the observable process $\{M_t^{\Dc, {\bm h}^{\tilde\phi}, \theta, \psi}\}_t$.

\begin{Proposition}\label{prop:grid-approximation} Let $(\theta^{\Dc, *}, \psi^{\Dc, *})$ be the solution to $\int_{\tilde\Phi} \Big(\E^e \Big[\int_0^T \eta_t^{\tilde\phi} d M_t^{\Dc,  {\bm h}^{\tilde\phi}, \theta, \psi}\Big]\Big) \Uc(d \tilde \phi) = 0$.
As $|\Dc| \to 0$, any convergent subsequence of $(\theta^{\Dc, *}, \psi^{\Dc, *})$ converges to the solution $(\theta^*, \psi^*)$ of \eqref{AMOC}.
\end{Proposition}
\begin{proof}By Theorem \ref{thm:grid-approximation}, 
$\lim_{|\Dc| \to 0} \int_{\tilde\Phi} \Big|\E^e\Big[\int_0^T \eta_t^{\tilde\phi} d \big(M_t^{\Dc,  {\bm h}^{\tilde\phi}, \theta, \psi} - M_t^{{\bm h}^{\tilde\phi}, \theta, \psi}\big)\Big]\Big|\Uc(d \tilde\phi) = 0$.
Combining this with Lemma 1 in \cite{JZ22b} yields the desired result.
\end{proof}

Using the temporal difference algorithm, we have the following updating rules of $\theta$ and $\psi$:
\begin{align} \label{orthogonal_updt1}
& \theta \leftarrow \theta +\alpha_\theta  \int_{\tilde\Phi} \E^e \left[\int_0^T \eta_t^{\tilde\phi} d M_t^{\Dc, {\bm h^{\tilde\phi}}, \theta, \psi}\right] \Uc(d \tilde\phi),\\
& \psi \leftarrow \psi + \alpha_{\psi}  \int_{\tilde\Phi} \E^e\left[\int_0^T  \zeta_t^{\tilde\phi} d M_t^{\Dc, {\bm h^{\tilde\phi}}, \theta, \psi} \right] \Uc(d \tilde\phi), \label{orthogonal_updt2}
\end{align}
where $\alpha_\theta$ and $\alpha_\psi$ are learning rates of $\theta$ and $\psi$, respectively. In particular, we choose the conventional test functions
\begin{align}\label{test_orthogonal}
&\eta_t^{\tilde\phi} =\frac{\partial J^\theta}{\partial \theta}\left(t, \mu_t^{\Dc, {\bm h}^{\tilde\phi}}\right),\;
\mathrm{and} \quad \zeta_t^{\tilde\phi} = \frac{\partial q^{0, \psi}}{\partial \psi} \left(t, \mu_t^{\Dc, {\bm h}^{\tilde\phi}}, {\bm h}^{\tilde\phi}\right).
\end{align}

Similar to \cite{weiyuyuan2023}, in numerical implementation we approximate the integral over the test policy parameter $\tilde\phi$ in \eqref{test_orthogonal} by a  Monte Carlo method and choose a sequence of test policies ${\bm h}^{\tilde\phi_1}$, $\ldots$, ${\bm h}^{\tilde\phi_M}$ with the same form yet different parameters from ${\bm \pi}^\psi$. Here, $\tilde\phi_1, \ldots, \tilde\phi_M$ are i.i.d. drawn from $\psi \cdot \Uc(p(n), q(n))$.  Recall that the social planner observes the conditional state distribution $\{\mu_s^{\Dc, {\bm h}^{\tilde\phi^m}}\}_s$ and the aggregated reward. When the sample data $(\mu_{t_k}^m, \hat r_{t_k}^m, \mu_{t_{k+1}}^m)_k$ generated by ${\bm h}^{\tilde\phi_m}$ with $t_k = k \Delta t$, $0 \leq k \leq K -1$, is available, we obtain a discretized version of martingale orthogonal condition \eqref{orthogonal_updt1}-\eqref{orthogonal_updt2} and the update rules of $\theta$ and $\psi$ with test functions \eqref{test_orthogonal} are written as
\begin{align}\label{critic-update}
\theta \leftarrow \theta + \alpha_\theta \frac{1}{M} \sum_{m=1}^M \Delta^m \theta, \; \psi \leftarrow \psi  + \alpha_\psi \frac{1}{M} \sum_{m=1}^M \Delta^m \psi,
\end{align}
where $G_{t_k: T}^m$, $\Delta^m \theta$ and $\Delta^m \psi$ are defined by
\begin{align*}
G_{t_k: T}^m & := J^\theta(t_{k+1}, \mu_{t_{k+1}}^m) - J^\theta(t_k, \mu_{t_k}^m) + \big(\hat r_{t_k}^m - \beta J^\theta(t_k, \mu_{t_k}^m) - q^{0, \psi}(t_k, \mu_{t_k}^m, {\bm h}^{\tilde\phi_m})\big)\Delta t,\\
\Delta^m \theta & := \sum_{k=0}^{K - 1} e^{-\beta t_k}\frac{\partial J^\theta}{\partial \theta}(t_k, \mu_{t_k}^m) G_{t_k: T}^m ,\;\Delta^m \psi  :=\sum_{k=0}^{K - 1} e^{-\beta t_k} \frac{\partial q^{0, \psi}}{\partial \psi}(t_k, \mu_{t_k}^m, {\bm h}^{\tilde\phi_m}) G_{t_k: T}^m .
\end{align*}
The offline optimal q-learning algorithm is reported in Algorithm \ref{algo:fixed-point}.
\begin{algorithm}[h]
\caption{Offline Optimal q-Learning Algorithm}
\textbf{Inputs}: initial conditional state distribution $\mu_0$,  horizon $T$, time step $\Delta t$, number of time steps $K$, number of test policies $M$, initial learning rates $\alpha_{\theta}$ and $\alpha_{\psi}$, functional forms of parameterized  value function $ J^{\theta}(\cdot,\cdot)$ satisfying $J^\theta(T, \cdot) = \hat g(\cdot)$, parameterized (unregularized) Iq-function $q^{0, \psi}(\cdot,\cdot,\cdot)$.

\textbf{Required program}: Environment simulator $(\mu',\hat r) = \textit{Environment}_{\Delta t}(t, \mu, {\bm h})$ that takes current time--state distribution pair $(t, \mu)$ and the policy ${\bm h}$ as inputs and generates conditional state distribution $\mu'$ at time $t+\Delta t$ and the aggregated reward $\hat r$ at time $t$ as outputs.


\textbf{Learning procedure}:
\begin{algorithmic}
\STATE Initialize $\theta$ and $\psi$.
\FOR{episode $n=1$ \TO $N$}
\STATE{Observe  the initial state distribution $\mu_0$ and store $\mu_{t_k}^m \leftarrow  \mu_0$.}
\FOR{$m=1$ \TO $M$}
\STATE{Draw $\tilde\phi^m$ from $\psi \cdot \mathcal{U}([p(n), q(n)])$ and set the test policy ${\bm h}^{\tilde\phi^m}$.}
\STATE{Initialize $k = 0$.
\WHILE{$k < K$} \STATE{
Apply the test policy ${\bm h}^{\tilde\phi^m}$ to the simulator $(\mu,\hat r) = Environment_{\Delta t}(t_k, \mu_{t_k}^m, {\bm h}^{\tilde\phi^m})$, and observe the new state distribution $\mu$ and the aggregated reward $\hat r$ as output. Store $\mu_{t_{k+1}}^m \leftarrow \mu$ and $\hat r_{t_k}^m \leftarrow \hat r$.

Update $k \leftarrow k + 1$.
}
\ENDWHILE	
}
\ENDFOR

Update $\theta$ and $\psi$ by $\theta \leftarrow \theta + \alpha_\theta \frac{1}{M} \sum_{m=1}^M \Delta^m \theta, \psi \leftarrow \psi + \alpha_{\psi} \frac{1}{M} \sum_{m=1}^M \Delta^m \psi$ by \eqref{critic-update}.

\ENDFOR
\RETURN $\theta$ and $\psi$.
\end{algorithmic}
\label{algo:fixed-point}
\end{algorithm}


\subsection{Actor-Critic q-learning algorithm}
To cope with the general case where the optimal policy as the two-layer fixed point to \eqref{twofix} may not admit an explicit form, we devise in this subsection an Actor-Critic algorithm. Specifically, we rely on Proposition \ref{thm:policy_improvement} (iii)  to design the Actor step for policy improvement (PI) to update the policy and utilize Theorem \ref{thm:unknown-q-function} to design the Critic step for policy evaluation (PE) to update the value function and the Iq-function.

\noindent
\underline{\textbf{Actor Step}} (or PI step): Let us consider a parameterized policy class $\Pi^\Phi = \{{\bm \pi}^\phi: \phi \in \Phi\}_{\phi \in \Phi}$. Let $\phi^n$ be the value of $\phi$ at the episode $n \in \N$. By Proposition \ref{thm:policy_improvement} (iii), if we update it to a new policy ${\bm \pi}^{\phi}$ such that
\begin{align}\label{actor-improve}
\E^e\left[\int_0^T e^{-\beta t} q^\gamma(t, \mu_t^{{\bm \pi}^{\phi}}, {\bm \pi}^{\phi^n}; {\bm \pi}^{\phi^n})dt\right] \leq \E^e\left[\int_0^T e^{-\beta t} q^\gamma(t, \mu_t^{{\bm \pi}^{\phi}}, {\bm \pi}^{\phi}; {\bm \pi}^{\phi^n})dt\right],
\end{align}
it then holds that $J(0, \mu; {\bm \pi}^{\phi}) \geq J(0, \mu; {\bm \pi}^{\phi^n})$. However, $\{\mu_t^{{\bm \pi}^{\phi}}\}_t$ is not known at the episode $n$ in advance.  We approximate $\mu_t^{{\bm \pi}^{\phi}}$ by $\hat\mu_t^n =(1 - \rho^n )\mu^{{\bm \pi}^{\phi^n}}_t + \rho^n \mu_t^{{{\bm h}^{{\tilde\phi}_0^n}}}$, where $\mu_t^{{\bm \pi}^{\phi^n}}$ and $\mu_t^{{\bm h}^{\tilde\phi_0^n}}$ are conditional state distribution generated by ${\bm \pi}^{\phi^n}$ and ${\bm h}^{\tilde\phi_0^n}$, respectively, $\tilde\phi^n_0$ is drawn from $\phi^n \cdot \Uc(p(n), q(n))$, and $\rho^n \in (0, 1)$ balances the weighting between the current estimate and the noise using the test policy ${\bm h}^{\tilde\phi_0^n}$. It is sufficient to solve the following optimization problem at the episode $n$
\begin{align}\label{actor-optimization-q-n}
\max_{\phi \in \Phi} \E^e\left[\int_0^T e^{-\beta t} q^\gamma(t, \hat\mu_t^n, {\bm \pi}^{\phi}; {\bm \pi}^{\phi^n})dt\right],
\end{align}
and to update $\phi$ using stochastic gradient ascent (SGA) rule. The gradient of the objective function in \eqref{actor-optimization-q-n} is given below.
\begin{Lemma}\label{actor-gradient}We have
{\footnotesize
\begin{align}\label{equ:actor-gradient}
&\nabla_{\phi}\E^e\left[\int_0^T e^{-\beta t} q^\gamma(t, \hat\mu_t^n, {\bm \pi}^{\phi}; {\bm \pi}^{\phi^n})dt\right]\\
=& \E^e\Big[\int_0^T e^{-\beta t} \int_{\R^d \times \Ac}\Big[\Big\{\frac{\delta q^0}{\delta {\bm h}}(t,\hat\mu_t^n, {\bm \pi}^{\phi}; {\bm \pi}^{\phi^n})(x, a) - \gamma \log {\bm \pi}^{\phi}(a|t, x, \hat\mu_t^n) \Big\}\nabla_{\phi}\log{\bm \pi}^{\phi}(a|t, x, \hat\mu_t^n)\Big] {\bm \pi}^{\phi}(a|t, x, \hat\mu_t^n)da\hat\mu_t^n(dx)\Big]. \nonumber
\end{align}}
\end{Lemma}
Suppose that the sample data $\{\hat\mu_{t_k}^{n}\}_{0 \leq k \leq K-1}$ generated by ${\bm \pi}^{\phi^n}$ and ${\bm h}^{\tilde\phi^n_0}$ are available.  Based on Lemma \ref{actor-gradient}, $\phi$ is updated incrementally by SGA  with $L-1$ inner iterations, $L \geq 1$, that
\begin{align}\label{inner-iteration}
\phi^{n, l + 1} &= \phi^{n, l} + \alpha_{\phi_{n, l}} \hat G_{\phi^{n, l}}^n.
\end{align}
where $\phi^{n, 0}: = \phi^n$, $\alpha_\phi$ is the learning rate, and $\hat G_\phi^n$ is a stochastic estimation of \eqref{equ:actor-gradient} that
\begin{align}
\hat G_\phi^n&= - \sum_{k=0}^{K -1} e^{-\beta t_k}\int_{\R^d \times \Ac}\Big\{\frac{\delta q^{0, \psi^n}}{\delta {\bm h}}(t_k, \hat\mu_{t_k}^{n}, {\bm \pi}^{\phi})(x, a) - \gamma \log {\bm \pi}^{\phi}(a|t_k, x, \hat\mu_{t_k}^{n})\Big\} \nonumber\\
& \;\;\; \nabla_{\phi} \log{\bm \pi}^{\phi}(a|t_k, x, \hat\mu_{t_k}^{n}) {\bm \pi}^\phi(a|t_k, x, \hat\mu_{t_k}^{n})da\hat\mu_{t_k}^{n}(dx).\label{estimate-gradient}
\end{align}
\begin{Remark}
It is expected that when $L$ is sufficiently large, ${\bm \pi}^{\phi^{n, L-2}}$ will get very close to the maximizer of $q^\gamma (\cdot; {\bm \pi}^{\phi^n})$. Due to the fact that the approximator of $q^\gamma$ with respect to $\phi$ is sophisticated and non-concave in general, the convergence of SGA \eqref{inner-iteration} for the optimization problem \eqref{actor-optimization-q-n} becomes complicated to check. We therefore analyze the convergence of gradient ascent \eqref{inner-iteration} for the LQ-MFC problem over an infinite horizon in section \ref{actor-convergence-LQ}, which may serve as a preliminary step towards a deeper understanding of \eqref{actor-optimization-q-n} in the general framework in our future study. In two examples in Section \ref{sec:examples}, our simulation experiments based on the inner iteration \eqref{inner-iteration} also show the satisfactory performance.
\end{Remark}

On the other hand,  it is expected that $\lim_{n \to \infty} q^\gamma(t, \mu, {\bm \pi}^{\phi^{n+1}}; {\bm \pi}^{\phi^n}) =0$.  Hence we shall modify the objective function in \eqref{actor-optimization-q-n} by adding another loss function, called the consistency loss function, to enforce the consistency condition $q^{\gamma, *}(t, \mu; {\bm \pi}^*)=0$ when $n$ is large that
\begin{align}\label{actor-optimization-q-constraint}
\max_{\phi \in \Phi} \int_0^T e^{-\beta t} \Big(w_o q^\gamma(t, \hat\mu_t^{n}, {\bm \pi}^{\phi}; {\bm \pi}^{\phi^n}) - \frac{w_c}{2} |q^\gamma(t, \hat\mu_t^{n}, {\bm \pi}^{\phi}; {\bm \pi}^{\phi^n})|^2\Big)dt,
\end{align}
where the parameters $w_o \geq 0$ and $w_c \geq 0$ stand for weights between the optimality loss function and the consistency loss function. At the beginning of the training ($n$ is small), we shall focus on improving the value function by setting $w_o \gg w_c$. After sufficiently many iterations ($n$ is large), we then set a large $w_c \gg w_o$ to enforce the consistency condition $\lim_{n \to \infty} q^\gamma(t, \mu, {\bm \pi}^{\phi^{n+1}}; {\bm \pi}^{\phi^n}) =0$: during this period, the distance between ${\bm \pi}^{{\phi}^{n+1}}$ and ${\bm \pi}^{\phi^n}$ becomes small and thus $q^\gamma(t, \hat\mu_t^{n}, {\bm \pi}^{\phi^{n+1}}; {\bm \pi}^{\phi^n})$ is approximately equal to $q^\gamma(t, \hat\mu_t^{n}, {\bm \pi}^{\phi^{n+1}}; {\bm \pi}^{\phi^{n+1}})$. Therefore, at each episode, after running $(L-1)$ inner iterations of \eqref{inner-iteration}, we also run one-step SGA for \eqref{actor-optimization-q-constraint} that
\begin{align}\label{actor-iteration-consistency}
\phi^{n+1} = \phi^{n, L-2} + \alpha_{\phi_{n, L-2}} \hat G_{\phi^{n, L-2}}^n,
\end{align}
where $\hat G_\phi$ is a stochastic estimation of the gradient of the objective function in \eqref{actor-optimization-q-constraint} that
\begin{align}
\hat G_\phi^n &= - \sum_{k=0}^{K -1} e^{-\beta t_k}\big(w_o - w_c q^{0, \psi}(t_k, \hat \mu_{t_k}^n, {\bm \pi}^{\phi})\big)\int_{\R^d \times \Ac}\Big\{\frac{\delta q^{0, \psi}}{\delta {\bm h}}(t_k, \hat\mu_{t_k}^{n}, {\bm \pi}^{\phi})(x, a) \nonumber\\
& \;\;\; - \gamma \log {\bm \pi}^{\phi}(a|t_k, x, \hat\mu_{t_k}^n)\Big\} \nabla_{\phi} \log{\bm \pi}^{\phi}(a|t_k, x, \hat\mu_{t_k}^n) {\bm \pi}^\phi(a|t_k, x, \hat\mu_{t_k}^n)da\hat\mu_{t_k}^n(dx).\label{estimate-gradient1}
\end{align}

{\begin{Remark}The gradients in \eqref{estimate-gradient} and \eqref{estimate-gradient1} involve the computation of the partial linear functional derivative  $\frac{\delta q^{0, \psi}}{\delta {\bm h}}$. In general, one can use cylindrical neural network functions (see \cite{phamwarin2022, phamwarin2023}) to approximate the Iq-function such that 
\begin{align}\label{cylinder-approximation-q}
q^{0, \psi}(t, \mu, {\bm h}) = \Phi(t, \langle\Psi, \mu \cdot {\bm h}\rangle),
\end{align}
where $\Phi: [0, T] \times \R^k \to \R$ and $\Psi: \R^d \times \Ac \to \R^k$ are two neural networks. 
It follows from the definition of partial linear functional derivative  (c.f.\cite{Renetal25}) that $\frac{\delta q^{0, \psi}}{\delta {\bm h}}(t, \mu, {\bm h})(x, a) = \partial_y\Phi(t, \langle\Psi, \mu \cdot {\bm h}\rangle) \Psi(x, a)$. Furthermore, the involved integrals with respect to ${\bm h}$ and $\mu$ in \eqref{estimate-gradient} and \eqref{cylinder-approximation-q} could be computed approximately with samples drawn from $\bm h$ and $\mu$ as in \cite{phamwarin2023}.
\end{Remark}
}

\vspace{0.2in}
\noindent
\underline{\textbf{Critic Step}} (or PE step): Given the updated policy ${\bm\pi}^{\phi}$ from the Actor step, the Critic step is essentially the same as the one in Algorithm \ref{algo:fixed-point} based on the averaged martingale orthogonality condition to update the value function $J^{\theta}$ and the Iq-function $q^{0, \psi}$ using iterations in \eqref{critic-update} except that the parameters of the test policies $\tilde\phi^1, \ldots, \tilde\phi^M$ are now i.i.d. drawn from $\phi \cdot \Uc(p(n), q(n))$.

In summary, for each episode $n$, the Actor-Critic q-learning algorithm consists of two steps, whose pseudo code is given in Algorithm \ref{algo:actor-critic}:
\begin{enumerate}[label={{\upshape(\roman*)}}]
\item The PE step that estimates the value function $J(;{\bm \pi}^\phi)$ and the Iq-function $q^0(; {\bm \pi}^\phi)$ via the temporal difference algorithm \eqref{critic-update}, where $J(;{\bm \pi}^\phi)$ and  $q^0(; {\bm \pi}^\phi)$ are estimated by parameterized function classes $\{J^\theta: \theta \in \Theta\}$ and $\{q^{0, \psi}: \psi \in \Psi\}$, respectively;
\item The PI step that has $L$ inner iterations and updates the parameter $\phi$ of the policy ${\bm \pi}^\phi$ using a stochastic version $\hat G_\phi$ of $G_\phi$ in \eqref{estimate-gradient1}, where $q^0(; {\bm \pi}^\phi)$ is replaced by the corresponding estimator $q^{0, \psi}$.
\end{enumerate}

\begin{algorithm}[!htb]
\caption{Offline Actor-Critic q-Learning Algorithm}
\textbf{Inputs}: initial state distribution $\mu_0$,  horizon $T$, time step $\Delta t$, number of mesh grids $K$, number of test policies $M$, the number of inner iterations for the Actor $L$, learning rates $\alpha_{\theta}$, $\alpha_{\psi}$ and $\alpha_{\phi}$, the parameterized  value function $ J^{\theta}(\cdot,\cdot)$ satisfying $J^\theta(T, \cdot) = \hat g(\cdot)$, the parameterized (unregularized) Iq-function $q^{0, \psi}(\cdot,\cdot,\cdot)$, and parametrized policy ${\bm \pi}^{\phi}$.

\textbf{Required program}: Environment simulator $(\mu',\hat r) = \textit{Environment}_{\Delta t}(t, \mu, {\bm h})$ that takes current time--state distribution pair $(t, \mu)$ and the policy ${\bm h}$ as inputs and generates state distribution $\mu'$ at time $t+\Delta t$ and the aggregated reward $\hat r$ at time $t$ as outputs.


\textbf{Learning procedure}:
\begin{algorithmic}
\STATE Initialize $\theta$, $\psi$ and $\phi$.

\FOR{episode $n=1$ \TO $N$}

\STATE{Observe  the initial state distribution $\mu_0$ and store $\mu_{t_k}^m \leftarrow  \mu_0$.}

\FOR{$m=1$ \TO $M$}
\STATE{Draw $\tilde\phi^m$ from $\phi \cdot \mathcal{U}([p(n), q(n)])$ and set the test policy ${\bm h}^{\tilde\phi^m}$.}

\FOR{$k=0$ \TO $K-1$}
\STATE{Apply the test policy ${\bm h}^{\tilde\phi^m}$, $0 \leq m \leq M$, to the simulator $(\mu,\hat r) = Environment_{\Delta t}(t_k, \mu_{t_k}^m, {\bm h}^{\tilde\phi^m})$, and observe the new state distribution $\mu$ and the aggregated reward $\hat r$ as output. Store $\mu_{t_{k+1}}^m \leftarrow \mu$ and $\hat r_{t_k}^m \leftarrow \hat r$.}

\STATE{Apply the policy ${\bm \pi}^\phi$ to the environment simulator and store $\hat\mu_{t_{k+1}} \leftarrow \mu$.}
\ENDFOR

\ENDFOR

\STATE{Update $\theta$ and $\psi$ in PE by $\theta \leftarrow \theta + \alpha_\theta \frac{1}{M} \sum_{m=1}^M \Delta^m \theta, \psi \leftarrow \psi + \alpha_{\psi} \frac{1}{M} \sum_{m=1}^M \Delta^m \psi$ by \eqref{critic-update}.}

\FOR{iteration $l = 0$ \TO $L-1$}
\STATE{Update $\phi$ in PI by $\phi \leftarrow \phi + \alpha_\phi \hat G_\phi$ according to \eqref{inner-iteration}}.
\ENDFOR
\STATE{Update $\phi$ by $\phi \leftarrow \phi + \alpha_\phi \hat G_\phi$ according to \eqref{actor-iteration-consistency}}.

\ENDFOR
\RETURN $\theta$, $\psi$ and $\phi$.
\end{algorithmic}
\label{algo:actor-critic}
\end{algorithm}

\section{Inner Iterations in LQ-MFC Over an Infinite Horizon}\label{actor-convergence-LQ}
In this section, we discuss the convergence of the gradient ascent iterations for the optimization problem \eqref{actor-optimization-q-n} in the framework of LQ-MFC problems over an infinite horizon.  A theoretical convergence or regret analysis of the Actor-Critic algorithm for LQ-MFC problem, as well as for general MFC problem over a finite horizon remains challenging and is left for future work.

\subsection{Infinite-horizon LQ-MFC problems}
Let us consider a LQ-MFC problem over an infinite horizon with $\Ac = \R^p$ and assume $n=m=1$ for simplicity. The coefficients of dynamics and the reward function are given by
\begin{align*}
b(x, \mu, a) &= B (x -\bar\mu) + (B + \bar B)\bar\mu +  C a, \\
\sigma(x, \mu, a) &=  D (x - \bar\mu) + (D + \bar D)\bar\mu + Fa,\\
\sigma_o( x, \mu, a) &= D_o(x - \bar\mu) + (D_o + \bar D_o)\bar\mu + F_o a,\\
r(x, \mu, a) & = x\trans M x + \bar\mu\trans \bar M \bar\mu + a \trans R a.
\end{align*}
Here $B$, $\bar B$, $D$, $\bar D$, $D_o$, $\bar D_o$, $M$ and $\bar M$ are deterministic matrices valued in $\R^{d \times d}$, $C, F, F_o$ are deterministic matrices valued in $\R^{d \times p}$, $R$ are deterministic matrices valued in $\R^{p \times p}$.
We may assume without loss of generality that $M$, $\bar M$ and $R$ are symmetric matrices.

Let $\bar\mu = \int_{\R^d} x \mu(dx)$ and ${\rm Var}(\mu)(\Lambda) = \int_{\R^d} (x - \bar\mu)\trans \Lambda (x -\bar\mu) \mu(dx)$ denote the mean and variance of $\mu$, respectively, for a symmetric matrix $\Lambda \in \R^{d \times d}$. We have the explicit expressions for the optimal value function and the optimal policy. The proof of Theorem \ref{thm:LQ-value-policy} is similar as that of Theorem 5.1 in \cite{Renetal25} and is thus omitted.
\begin{Theorem}\label{thm:LQ-value-policy}Under the condition $M \preceq 0,\; M + \bar M \preceq 0,\; R \preceq -\delta I_q$, the optimal value function $J^*$ admits the form
\begin{align}\label{LQ-J*-formulation}
J^*(\mu) = {\rm Var}(\mu)(\Lambda^*) + \bar\mu\trans \Gamma^* \bar\mu + \chi^*,
\end{align}
and the optimal policy ${\bm \pi}^*$ is a Gaussian policy that
$${\bm \pi}^*(\cdot|x, \mu) = \Nc\Big(- (U^*)^{ -1} S^* (x - \bar\mu) -(V^*)^{-1} Z^* \bar\mu, - \frac{\gamma}{2} (U^*)^{-1}\Big),$$
with
$\Lambda^*, \Gamma^*$ and $\chi^*$ satisfying
\begin{align}\label{ODELambda-optimal}
-\beta \Lambda^* + M + D\trans \Lambda^* D + D_o\trans \Lambda^* D_o + B\trans \Lambda^* + \Lambda^* B - (S^*)\trans (U^*)^{-1} S^*=0,
\end{align}
\begin{align}\label{ODEGamma-optimal}
& - \beta \Gamma^* + M+ \bar M + \big(D +\bar D)\trans \Lambda^*\big(D + \bar D) +  (\bar D_o + D_o)\trans \Gamma^*(\bar D_o + D_o)\\
& \hspace{3cm}+ (B + \bar B)\trans \Gamma^*+ \Gamma^* (B + \bar B) - (Z^*)\trans (V^*)^{-1}Z^* =0, \nonumber
\end{align}
\begin{align}\label{ODEchi-optimal}
 -\beta\chi^*
 + \frac{\gamma p}{2}\log \big(\gamma\pi\big) - \frac{\gamma}{2} \log \big({\rm det}(-U^*)\big)  =0,
\end{align}
and
$U:= U_\Lambda$, $V: = V_{\Gamma, \Lambda}$, $S: = S_{\Lambda}$, $Z: = Z_{\Gamma, \Lambda}$, with
\begin{align}\label{UVSZ}
\left\{
\begin{array}{lll}
U &= F\trans \Lambda F +  F_o \trans \Lambda F_o + R,\\
V & = F\trans \Lambda F  + F_o \trans \Gamma F_o+ R,\\
S & = C\trans \Lambda + F\trans \Lambda D + F_o\trans \Lambda D_o,\\
Z & = C\trans\Gamma + F\trans \Lambda\big(D + \bar D\big) + F_o \trans \Gamma\big( D_o+ \bar D_o\big).
\end{array}
\right.
\end{align}
\end{Theorem}

By Theorem \ref{thm:LQ-value-policy}, the optimal policy is Guassian. From this, we restrict our attention to the policy set parameterized with the parameter $\phi = (K, \bar K, \Sigma)$
\begin{align*}
\Pi= \big\{\phi \in \R^{p \times d} \times \R^{p \times d} \times \R^{d \times d}: {\bm \pi}^\phi(\cdot|x, \bar\mu) =\Nc(K(x - \bar\mu) + \bar K \bar\mu, \Sigma), \; \Sigma \succ 0,\; \Sigma\trans = \Sigma\big\}.
\end{align*}

Given any parameter $\phi = (K, \bar K, \Sigma) \in \Pi$, we can easily obtain the expression of  $J(\mu; {\bm \pi}^\phi)$ by applying the dynamic programming equation in \eqref{equ:dynamic-programming-equation} as well as the explicit expression of Iq-function in the next result.

\begin{Lemma}\label{lemma:value-function}The value function $J(\cdot; {\bm \pi}^\phi)$ takes the form $J(\mu; {\bm \pi}^\phi) = {\rm Var}(\mu)(\Lambda_{K}) + \bar\mu\trans \Gamma_{\bar K}\bar\mu + \chi_{\Sigma}$, where $\Lambda_{K}$, $\Gamma_{\bar K}$ and $\chi_{\Sigma}$ satisfy
\begin{align}\label{ODELambda}
& -\beta \Lambda_{K} + D\trans \Lambda_{K} D  + D_o\trans \Lambda_{K} D_o +  B\trans \Lambda_{K} + \Lambda_{K} B + M + 2 S_K\trans K + K\trans U_K K = 0, \\
& -\beta \Gamma_{\bar K} +  \big(D +\bar D)\trans \Lambda_{K}\big(D + \bar D) +  (D_o + \bar D_o)\trans \Gamma_{\bar K}(D_o + \bar D_o) +  (B + \bar B)\trans \Gamma_{\bar K}  \nonumber\\
& + \Gamma_{\bar K}(B + \bar B) + M + \bar M  + 2Z_{K, \bar K} \trans \bar K + \bar K \trans V_{K, \bar K} \bar K = 0, \label{ODEGamma}\\
&-\beta \chi_\Sigma + \frac{\gamma p}{2}\Big(\log (2\pi) + 1 \Big) + \frac{\gamma}{2} \log \big( {\rm det}(\Sigma)\big) + {\rm Tr}\big(U \Sigma\big) = 0.\label{ODEchi}
\end{align}
In addition, for any ${\bm h}^{\tilde\phi} \in \Pi$ with $\tilde\phi = (\tilde K, \tilde {\bar K}, \tilde \Sigma)$, Iq-function $q^\gamma(\mu, {\bm h}; {\bm \pi}^\phi)$ takes the form
\begin{align}\label{LQ-q}
& q^\gamma(\mu, {\bm h}^{\tilde\phi}; {\bm \pi}^\phi) = q^{\gamma, \phi}_1(\mu, \tilde K) + q^{\gamma, \phi}_2(\mu, \tilde {\bar K}) + q^{\gamma, \phi}_3(\tilde\Sigma),
\end{align}
where $\Lambda_K$, $\Gamma_{\bar K}$ and $\chi_\Sigma$ satisfy \eqref{ODELambda}-\eqref{ODEchi}, $q^\phi_1(\tilde K)$, $q^\phi_2(\tilde{\bar K})$ and $q^\phi_3(\tilde\Sigma)$ are respectively given by
\begin{align*}
q^{\gamma, \phi}_1(\mu, \tilde K)  =& {\rm Var}(\mu)\Big(-\beta \Lambda_K + D\trans \Lambda_K D  + D_o\trans \Lambda_K D_o +  B\trans \Lambda_K + \Lambda_K B + M + 2 S_K\trans \tilde K + \tilde K\trans U_K \tilde K\Big),\\
q^{\gamma, \phi}_2(\mu, \tilde{\bar K})  =& \bar\mu\trans \Big(-\beta \Gamma_{\bar K} +  \big(D +\bar D)\trans \Lambda_K \big(D + \bar D) +  (D_o + \bar D_o)\trans \Gamma_{\bar K}(D_o + \bar D_o)\\
&+  (B + \bar B)\trans \Gamma_{\bar K} + \Gamma_{\bar K}(B + \bar B) + M + \bar M  + 2 Z_{K, \bar K}\trans \tilde{\bar K} + \tilde{\bar K} \trans V_{K, \bar K} \bar K\Big) \bar\mu, \nonumber\\
q^{\gamma, \phi}_3(\tilde\Sigma)  =& -\beta \chi_\Sigma + \frac{\gamma p}{2}\big(\log (2\pi) + 1\big) + \frac{\gamma}{2}\log \big({\rm det}(\tilde\Sigma)\big) + {\rm Tr}\big(U^n \tilde\Sigma\big).
\end{align*}
\end{Lemma}

\subsection{Approximated Iq-function}
Note that $q^\gamma(\mu, {\bm h}^{\tilde\phi}; {\bm  \pi}^\phi)$ in \eqref{LQ-q} depends on $\Lambda_K$, $\Gamma_{\bar K}$ and $\chi_{\Sigma}$, and hence depends on model parameters and it is unknown. To analyze the convergence of the inner actor iteration in the model-free setting, we need to approximate $q^\gamma(\mu, {\bm \pi}^{\phi}; {\bm \pi}^{\phi^n})$ by  $q^{\gamma, \psi^n}(\mu, {\bm \pi}^{\phi})$ for each fixed outer iteration $n \in \N$. By the structure of Iq-function in \eqref{LQ-q}, we assume $q^{\gamma, \psi^n}(\mu, {\bm \pi}^{\phi})$ takes the form, with $\psi^n = (\mathsf{\Lambda}^n, \mathsf{\Gamma}^n, \mathsf{\chi}^n, S^n, U^n, Z^n, V^n)$
\begin{align}\label{LQ-approximate-q}
q^{\gamma, \psi^n}(\mu, {\bm \pi}^{\phi}) = q^{\gamma, \psi^n}_1(\mu, K) + q^{\gamma, \psi^n}_2(\mu, {\bar K}) + q^{\gamma, \psi^n}_3(\Sigma),
\end{align}
where
\begin{align*}
q^{\gamma, \psi^n}_1(\mu, K) & = {\rm Var}(\mu) \Big(\mathsf{\Lambda}^n + 2  S^n K +  K\trans U^n K\Big),\\
q^{\gamma, \psi^n}_2(\mu, {\bar K}) & = \bar\mu\trans \Big(\mathsf{\Gamma}^n + 2 (Z^n)\trans {\bar K} + {\bar K}\trans V^n {\bar K}\Big) \bar\mu,\\
q^{\gamma, \psi^n}_3(\Sigma) & = \mathsf{\chi}^n + \frac{\gamma}{2}\log \big({\rm det}(\Sigma)\big) + {\rm Tr}\big(U^n \Sigma\big).
\end{align*}
Recall that the inner iteration \eqref{actor-optimization-q-n} for each fixed $n \in \N$ is to solve the optimization problem:
\begin{align}\label{inner-optimization}
\max_{\phi \in \Pi} q^{\beta, \gamma, \psi^n}(\mu, {\bm \pi}^\phi) : &= \max_{\phi \in \Pi} \E^e\Big[\int_0^{+\infty} e^{-\beta t} \Big(q_1^{\gamma, \psi^n}(\hat\mu_t^n, K) + q_2^{\gamma, \psi^n}(\hat\mu_t^n, \bar K) + q_3^{\gamma, \psi^n}(\Sigma)\Big) dt\Big]\\
& =: q_1^{\beta, \gamma, \psi^n}(\mu, K) + q_2^{\beta, \gamma, \psi^n}(\mu, \bar K) + q_3^{\beta, \gamma, \psi^n}(\Sigma),\nonumber
\end{align}
where $\hat\mu_t^n = (1 - \rho^n )\mu^{{\bm \pi}^{\phi^n}}_t + \rho^n \mu_t^{{{\bm h}^{{\tilde\phi}_0^n}}}$ with the initial value $\hat \mu_0^n = \mu$. The next result directly follows whose proof is omitted.
\begin{Lemma}\label{lem:unique-maximizer}
Assume that $U^n$ and $V^n$ are negative definite. For given $\psi^n$, the approximated Iq-function $q^{\beta, \gamma, \psi^n}(\mu, {\bm \pi}^{\phi})$  is concave in $\phi$. Moreover, the unique maximizer $\phi^{n+1}$ in \eqref{inner-optimization} is
\begin{align}\label{n+1-policy}
K^{n + 1} =- (U^n)^{-1} S^n, \; \bar K^{n +1} = - (V^n)^{-1} Z^n, \; \Sigma^{n+1}= - \frac{\gamma}{2} (U^n)^{-1}.
\end{align}
\end{Lemma}

For any square matrix $Z \in \R^{m \times m}$, we denote $\sigma_{\min}(Z)$,  $\|Z\|$ and $\|Z\|_F$ the minimum eigenvalue value, the spectral norm and Frobenius norm  of $Z$ , respectively. A function $f: \R^{m \times n} \to \R$ is said to be $\Upsilon$-smooth if it is differentiable and $\|\nabla f(X) - \nabla f(X')\|_F \leq \Upsilon \|X - X'\|_F$. To implement the gradient ascent update, we first derive the gradient of the objective function. Some direct computations lead to the next result.

\begin{Lemma}\label{derivativeq}We have
\begin{align*}
\nabla_{K}q^{\beta, \gamma, \psi^n}(\mu, {\bm \pi}^{\phi})& = 2\Big(S^n + U^n K\Big)C_{\beta, \mu}, \\
\nabla_{{\bar K}} q^{\beta, \gamma, \psi^n}(\mu, {\bm \pi}^{\phi}) &  = 2 \Big(Z^n + V^n \bar K\Big)C_{\beta, \bar\mu}, \\
\nabla_{\Sigma}q^{\beta, \gamma, \psi^n}({\bm \pi}^{\phi})& = \frac{1}{1-\beta}\big(\frac{\gamma} {2}\Sigma^{-1} + U^n\big),
\end{align*}
where $C_{\beta, \mu}$ and $C_{\beta, \bar\mu}$ are defined by
\begin{align*}
C_{\beta, \mu}: = \E^e\left[\int_0^{+\infty} \int_{\R^d}e^{-\beta t}(x-\bar{\hat\mu}_t^n)(x - \bar{\hat\mu}_t^n)\trans \hat\mu^n_t(dx)dt\right],\; C_{\beta, \bar\mu}:= \E^e\left[\int_0^{+\infty} e^{-\beta t}\bar{\hat\mu}_t \bar{\hat\mu}_t \trans dt\right].
\end{align*}
Moreover, if $\Sigma \succeq a I_p$ for some $a > 0$,$q^{\gamma, \psi^n}(\mu, {\bm \pi}^{\phi})$ is $\Upsilon$-smooth, where $\Upsilon = \max\{2 \|U^n\| \cdot \|C_{\beta, \mu}\|, 2 \|V^n\| \cdot \|C_{\beta, \bar\mu}\|, \frac{\gamma}{2a^2(1-\beta)}\}$.
\end{Lemma}

\subsection{Convergence of inner fixed point iteration}
Consider the gradient ascent for the parameters $K$, $\bar K$ and $\Sigma$ that
\begin{align}
K^{n, (l + 1)}&= K^{n, (l)} + \alpha_{K} \nabla_{K}q^{\beta, \gamma, \psi^n}_1(\mu, K), \nonumber\\
\bar K^{n, (l+1)} &= \bar K^{n, (l)} + \alpha_{\bar K} \nabla_{{\bar K}} q^{\beta, \gamma, \psi^n}_2(\mu, \bar K), \label{actor:PG}\\
\Sigma^{n, (l + 1)} &= \Sigma^{n, (l)}  + \alpha_{\Sigma} \nabla_{\Sigma}q^{\beta, \gamma, \psi^n}_3(\Sigma). \nonumber
\end{align} From the explicit expressions in Lemma \ref{derivativeq}, the equation \eqref{actor:PG} is equivalent to
\begin{align}
K^{n, (l + 1)}&= K^{n, (l)} + 2\alpha_{K} \Big(S^n + U^n K^{n, (l)}\Big)C_{\beta, \mu}, \nonumber\\
\bar K^{n, (l+1)} &= \bar K^{n, (l)} + 2 \alpha_{\bar K} \Big(Z^n + V^n \bar K^{n, (l)}\Big) C_{\beta, \bar\mu}, \label{actor:PG-explicit}\\
\Sigma^{n, (l + 1)} &= \Sigma^{n, (l)}  + \frac{\alpha_{\Sigma}}{1 - \beta} \Big(\frac{\gamma}{2}(\Sigma^{n, (l)})^{-1} + U^n\Big). \nonumber
\end{align}

With a suitable choice of $\alpha_\Sigma$ and initial value $\Sigma^{n, (0)}$, the next result shows that the iterative update $\eqref{actor:PG-explicit}$ for the covariance $\Sigma$ is invariant in the following sense.
\begin{Lemma}\label{lemma:invariance-Sigma}
Fix $n \in \N$. Given $a \geq \frac{\gamma}{2} \|U^n\|$ and $b \geq \frac{2 a^2}{\gamma \sigma_{\min}(-U^n)}$, take $aI_p \preceq \Sigma^{n, (0)} \preceq b I_p$. Let $\Sigma^{n, (l)}$ be $l$-th update \eqref{actor:PG-explicit} with $\alpha_\Sigma = \frac{2a^2(1 - \beta)}{\gamma}$, then $aI_p \preceq \Sigma^{n, (l)} \preceq b I_p$ for every $l=0, \ldots.$
\end{Lemma}

\begin{proof}First, we show that $a I_p \preceq \Sigma^{n, (1)} \preceq bI_p$. Define the function $g(x) = x + \frac{\gamma \alpha_\Sigma }{2 x} = x + \frac{a^2}{x}$. Then, $g$ is increasing on $[a, \infty)$. As $\Sigma^{n, (0)} \succeq a I_p$, then
\begin{align*}
\Sigma^{n, (1)}  &= \Sigma^{n, (0)} + a^2 (\Sigma^{n, (0)})^{-1} + \frac{\gamma}{2} U^n \succeq I_p + a I_p + \frac{\gamma}{2} U^n \succeq a I_p.
\end{align*}
By using $\Sigma^{n, (0)} \preceq b I_p$, we get that
\begin{align*}
\Sigma^{n, (1)} = \Sigma^{n, (0)} + a^2 (\Sigma^{n, (0)})^{-1} + \frac{\gamma}{2} U^n \preceq bI_p + \frac{a^2}{b} I_p + \frac{\gamma}{2} U^n \preceq bI_p + \frac{\gamma}{2} \sigma_{\min}(-U_n)I_p + \frac{\gamma}{2} U_n \preceq bI_p.
\end{align*}
By induction, we get $aI_p \preceq \Sigma^{n, (l)} \preceq b I_p$.
\end{proof}

We also make the following assumption on $C_{\beta, \mu}$ and $C_{\beta, \bar\mu}$ defined in Lemma \ref{derivativeq}.

\begin{Assumption}\label{ass:PG-contraction}
The symmetric matrices $C_{\beta, \mu}$ and $C_{\beta, \bar\mu}$ are positive definite, and $a I_p \preceq \Sigma^{n, (0)} \preceq b I_p$ with $a$ and $b$ satisfying $a \geq \frac{\gamma}{2} \|U^n\|$ and $b \geq \frac{2 a^2}{\gamma \sigma_{\min}(-U^n)}$.
\end{Assumption}

 Recall that a function $f$ is $\upsilon$-strongly concave with $\upsilon>0$ if it is differentiable and $-f$ is strongly convex, that is,
\begin{align*}
-f(X') \geq -f(X) + {\rm Tr}\big(\nabla (-f)(X)\trans (X' - X)\big) + \upsilon \|X' - X\|_F^2.
\end{align*}
By Assumption \ref{ass:PG-contraction}, the minimum eigenvalue values of $C_{\beta, \mu}$ and $C_{\beta, \bar\mu}$ are positive and hence $q^n$ are $\upsilon$-strongly concave and we can easily derive the next result.

\begin{Lemma}\label{lemma:strongly-concave}Under Assumption \ref{ass:PG-contraction}, let $\upsilon=\min\big\{\sigma_{\min}(-U^n) \sigma_{\min}(C_{\beta, \mu}), \sigma_{\min}(-V^n) \sigma_{\min}(C_{\beta, \bar\mu}), \frac{\gamma}{2b^2} \big\}$, we have that
\begin{align*}
{\rm Tr}\Big(\big(\nabla_K q_1^{\beta, \gamma, \psi^n}(\mu, K') - \nabla_K  q_1^{\beta, \gamma, \psi^n}(\mu, K)\big)\trans (K' - K)\Big]\Big) &\leq - \upsilon\|K' - K\|_F^2,\\
{\rm Tr}\Big(\big(\nabla_{\bar K} q_2^{\beta, \gamma, \psi^n}(\mu, \bar K') - \nabla_{\bar K} q_2^{\beta, \gamma, \psi^n}(\mu, \bar K)\big)\trans (\bar K' - \bar K)\Big]\Big) &\leq - \upsilon \|\bar K' - \bar K\|^2_F,\\
{\rm Tr}\Big(\big(\nabla_\Sigma q_3^{\beta, \gamma, \psi^n}(\Sigma') - \nabla_\Sigma q_3^{\beta, \gamma, \psi^n}(\Sigma)\big)\trans (\Sigma' - \Sigma)\Big) &\leq - \upsilon \|\Sigma' - \Sigma\|_F^2.
\end{align*}
Furthermore, $q^{\beta, \gamma, \psi^n}(\mu)$ is $\upsilon$-strongly concave.
\end{Lemma}

When $q^{\beta, \gamma, \psi^n}$ is $\Upsilon$-strongly smooth and $\upsilon$-strongly concave, we have the following convergence of inner fixed point iteration.
\begin{Proposition}\label{strong-convergence}
 For fixed $n \in \N$, Let $(K^{n, (0)}, \bar K^{n, {0}}, \Sigma^{n, (0)}) \in \Ac$ with $aI_p \preceq \Sigma^{n, (0)} \preceq b I_p$. Under Assumption \ref{ass:PG-contraction}, there exists some $\eta \in (0, 1)$ such that, after sufficiently large $L$ iterations, the iterative algorithm \eqref{actor:PG-explicit} with $\alpha_K = \alpha_{\bar K} = \alpha_\Sigma = \frac{\upsilon}{\Upsilon}$ achieves the following performance bound
 \begin{align*}
0 &\leq \max_{\phi \in \Pi} q^{\beta, \gamma, \psi^n}(\mu, {\bm \pi}^{\phi}) - q^{\beta, \gamma, \psi^n}(\mu, {\bm \pi}^{\phi^{n, (L)}})\leq \big(1- \eta\big)^L \Big(q^{\beta, \gamma, \psi^n}(\mu, {\bm \pi}^{\phi^{n+1}}) - q^{\beta, \gamma, \psi^n}(\mu, {\bm \pi}^{\phi^{n}})\Big),
 \end{align*}
and
\begin{align*}
& \|K^{n + 1} - K^{n, (L)}\|_F^2 + \|\bar K^{n+1} - \bar K^{n, (L)}\|^2_F + \|\Sigma^{n+1} - \Sigma^{n, (L)}\|_F^2\\
\leq & (1 - \eta)^L \Big(\|K^{n + 1} - K^{n, (0)}\|_F^2 + \|\bar K^{n+1} - \bar K^{n, (0)}\|^2_F + \|\Sigma^{n+1} - \Sigma^{n, (0)}\|_F^2\Big),
\end{align*}
where $(K^{n+1}, \bar K^{n+1}, \Sigma^{n+1})$ is the unique maximizer to the optimization problem \eqref{inner-optimization}, see Lemma \ref{lem:unique-maximizer}.

\end{Proposition}

\begin{proof}The proof is standard and we only provide the proof of $q^{\beta, \gamma, \psi^n}_3$ for the completeness. By Lemma \ref{lemma:strongly-concave}, we have
\begin{align}
q_3^{\beta, \gamma, \psi^n}(\Sigma') &= q_3^{\beta, \gamma, \psi^n}(\Sigma) + {\rm Tr}\big(\nabla_\Sigma q_3^{\beta, \gamma, \psi^n}(\Sigma)\trans \big(\Sigma' - \Sigma\big)\big) \nonumber\\
& + \int_0^1 {\rm Tr}\Big(\big(\nabla_\Sigma q_3^{\beta, \gamma, \psi^n}((1 - t)\Sigma + t \Sigma') - \nabla_\Sigma q_3^{\beta, \gamma, \psi^n}(\Sigma)\big)\trans (\Sigma' - \Sigma)\Big)dt \nonumber\\
& \leq q_3^{\beta, \gamma, \psi^n}(\Sigma) + {\rm Tr}\big(\nabla_\Sigma q_3^{\beta, \gamma, \psi^n}(\Sigma)\trans \big(\Sigma' - \Sigma\big)\big) - \frac{\upsilon}{2} \|\Sigma' - \Sigma\|_F^2\label{strongly-concave}\\
& = q_3^{\beta, \gamma, \psi^n}(\Sigma) -\frac{1}{2} \|\big(\sqrt{\upsilon}(\Sigma' - \Sigma) - \frac{1}{\sqrt{\upsilon}} \nabla_\Sigma q_3^{\beta, \gamma, \psi^n}(\Sigma)\big)\|_F^2 + \frac{1}{2\upsilon}\|\nabla_\Sigma q_3^{\beta, \gamma, \psi^n}(\Sigma)\|_F^2 \nonumber\\
& \leq q_3^{\beta, \gamma, \psi^n}(\Sigma)+  \frac{1}{2\upsilon}\|\nabla_\Sigma q_3^{\beta, \gamma, \psi^n}(\Sigma)\|_F^2 .
\label{equ:strongly-concave}
\end{align}
Substituting $\Sigma' = \Sigma^{n+1}$ and $\Sigma = \Sigma^{n, (l)}$ into  \eqref{equ:strongly-concave}, we have
\begin{align}\label{proof:concave-inequality}
\frac{1}{2\upsilon}\Big(q_3^{\beta, \gamma, \psi^n}(\Sigma^{n+1}) - q_3^{\beta, \gamma, \psi^n}(\Sigma^{n, (l)})\Big) \leq {\rm Tr}\big(\nabla_\Sigma q_3^{\beta, \gamma, \psi^n}(\Sigma^{n, (l)})\trans\nabla_\Sigma q_3^{\beta, \gamma, \psi^n}(\Sigma^{n, (l)})\big).
\end{align}
By fundamental theorem of calculus, we have that, for $l =0, 1, \ldots, L-1$,
\begin{align}
& q_3^{\beta, \gamma, \psi^n}(\Sigma^{n, (l + 1)}) \nonumber\\
= & q_3^{\beta, \gamma, \psi^n}(\Sigma^{n, (l)}) + {\rm Tr}\Big(\nabla_\Sigma q_3^{\beta, \gamma, \psi^n}(\Sigma^{n, l})\trans (\Sigma^{n, (l + 1)} - \Sigma^{n, (l)})\Big) \nonumber\\
& + \int_0^1 {\rm Tr}\Big(\big(\nabla_\Sigma q_3^{\beta, \gamma, \psi^n}(\Sigma^{n, (l)} + t (\Sigma^{n, (l+1)} - \Sigma^{n, (l)})) - \nabla_\Sigma q_3^{\beta, \gamma, \psi^n}(\Sigma^{n, l})\big)\trans (\Sigma^{n, (l + 1)} - \Sigma^{n, (l)})\Big) dt \nonumber\\
\geq  & q_3^{\beta, \gamma, \psi^n}(\Sigma^{n, (l)}) + \alpha_\Sigma {\rm Tr}\Big(\nabla_\Sigma q_3^{\beta, \gamma, \psi^n}(\Sigma^{n, l})\trans \nabla_\Sigma q_3^{\beta, \gamma, \psi^n}(\Sigma^{n, l})\Big) \nonumber\\
& -\frac{1}{2} \cdot \Upsilon \cdot \alpha_\Sigma^2 {\rm Tr}\Big(\nabla_\Sigma q_3^{\beta, \gamma, \psi^n}(\Sigma^{n, l})\trans \nabla_\Sigma q_3^{\beta, \gamma, \psi^n}(\Sigma^{n, l})\Big) \nonumber\\
= &  q_3^{\beta, \gamma, \psi^n}(\Sigma^{n, (l)}) + \frac{1}{2}\alpha_\Sigma {\rm Tr}\Big(\nabla_\Sigma q_3^{\beta, \gamma, \psi^n}(\Sigma^{n, (l)})\trans \nabla_\Sigma q_3^{\beta, \gamma, \psi^n}(\Sigma^{n,(l)})\Big).\label{proof-concave-inequality}
\end{align}
Plugging \eqref{proof:concave-inequality} into \eqref{proof-concave-inequality}, we obtain that
\begin{align*}
q_3^{\beta, \gamma, \psi^n}(\Sigma^{n, (l + 1)}) &\geq q_3^{\beta, \gamma, \psi^n}(\Sigma^{n, (l)}) + \frac{1}{2}\alpha_\Sigma {\rm Tr}\Big(\nabla_\Sigma q_3^{\gamma, \psi^n}(\Sigma^{n, l})\trans \nabla_\Sigma q_3^{\beta, \gamma, \psi^n}(\Sigma^{n,l})\Big)\\
&\geq  q_3^{\beta, \gamma, \psi^n}(\Sigma^{n, (l)}) + \frac{\upsilon}{\Upsilon} \Big(q_3^{\beta, \gamma, \psi^n}(\Sigma^{n+1}) - q_3^{\beta, \gamma, \psi^n}(\Sigma^{n, (l)})\Big).
\end{align*}
Subtracting $q_3^{\beta, \gamma, \psi^n}(\Sigma^{n+1})$ on both sides and rearranging the terms lead to
\begin{align*}
q_3^{\beta, \gamma, \psi^n}(\Sigma^{n+1}) - q_3^{\beta, \gamma, \psi^n}(\Sigma^{n, (l + 1)}) \leq \big(1 - \frac{\upsilon}{\Upsilon}\big)\Big(q_3^{\beta, \gamma, \psi^n}(\Sigma^{n+1}) - q_3^{\beta, \gamma, \psi^n}(\Sigma^{n, (l)})\Big).
\end{align*}
Moreover, we have that
\begin{align*}
& \|\Sigma^{n+1} - \Sigma^{n, (l+1)}\|_F^2\\
= &\|\Sigma^{n+1} - \Sigma^{n, (l)}\|_F^2 - 2 \alpha_\Sigma {\rm Tr}\Big(\nabla_\Sigma q_3^{\beta, \gamma, \psi^n}(\Sigma^{n, (l)})\trans \big(\Sigma^{n+1} - \Sigma^{n, (l)}\big)\Big) + \alpha_\Sigma^2 \|\nabla_\Sigma q_3^{\beta, \gamma, \psi^n}(\Sigma^{n, (l)})\|_F^2\\
\leq & (1 - \frac{\upsilon}{\Upsilon}\big)\|\Sigma^{n+1} - \Sigma^{n, (l)}\|_F^2 - 2 \alpha_\Sigma\big(q_3^{\beta, \gamma, \psi^n}(\Sigma^{n+1}) - q_3^{\beta, \gamma, \psi^n}(\Sigma^{n, (l)})\big) + \alpha_\Sigma^2 \|\nabla_\Sigma q_3^{\beta, \gamma, \psi^n}(\Sigma^{n, (l)})\|_F^2\\
\leq & (1 - \frac{\upsilon}{\Upsilon}\big)\|\Sigma^{n+1} - \Sigma^{n, (l)}\|_F^2  + \big(2\alpha_\Sigma - 2 \alpha_\Sigma\big) \big(q_3^{\beta, \gamma, \psi^n}(\Sigma^{n+1}) - q_3^{\beta, \gamma, \psi^n}(\Sigma^{n, (l)})\big),
\end{align*}
where in the last inequality, we have used \eqref{strongly-concave} and \eqref{proof-concave-inequality}. Repeating the same arguments for $q^{\beta, \gamma, \psi^n}_1$ and $q^{\beta, \gamma, \psi^n}_2$, we obtain the desired result.
\end{proof}

We then readily conclude the next result.
\begin{Corollary}\label{cor:LQconvergence}
For fixed $n \in \N$, Let $(K^{n, (0)}, \bar K^{n, {0}}, \Sigma^{n, (0)}) \in \Ac$ with $aI_p \preceq \Sigma^{n, (0)} \preceq b I_p$. Under Assumption \ref{ass:PG-contraction}, there exists some $\eta \in (0, 1)$ such that, after sufficiently large $L$ iterations, the iterative algorithm \eqref{actor:PG-explicit} with $\alpha_K = \alpha_{\bar K} = \alpha_\Sigma = \frac{\upsilon}{\Upsilon}$ achieves
 \begin{align*}
 q^{\beta, \gamma}(\mu, {\bm \pi}^{\phi^{n, (L)}}; {\bm \pi}^{\phi^n}) \geq  \max_{\phi \in \Pi} q^{\beta, \gamma}(\mu, {\bm \pi}^{\phi}; {\bm \pi}^{\phi^n}) - 2 \Delta - \big(1- \eta\big)^L \Big(q^{\beta, \gamma, \psi^n}(\mu, {\bm \pi}^{\phi^{n+1}}) - q^{\beta, \gamma, \psi^n}(\mu, {\bm \pi}^{\phi^{n}})\Big),
 \end{align*}
 where $\Delta: = \sup_{\phi}|q^{\beta, \gamma, \psi^n}(\mu, {\bm \pi}^{\phi}) - q^{\beta, \gamma}(\mu, {\bm \pi}^{\phi}; {\bm \pi}^{\phi^n})|$.
\end{Corollary}

By Corollary \ref{cor:LQconvergence}, if $\Delta$ is small, that is, $q^{\beta, \gamma}(\mu, {\bm \pi}^\phi; {\bm \pi}^{\phi^n})$ is well approximated by $q^{\beta, \gamma, \psi^n}(\mu, {\bm \pi}^\phi)$, then after sufficiently large $L$ iterations, ${\bm \pi}^{\phi^{n, (L)}}$ is a near-optimal policy of $\sup_{\phi} q^{\beta, \gamma}(\mu, {\bm \pi}^\phi; {\bm \pi}^{\phi^n})$.

\section{Numerical Examples}\label{sec:examples}

\subsection{LQ-MFC problem}
Let us consider a linear model of the wealth process controlled by the social planner on behalf of the population in a mean-field setting, whose dynamics    
under discretely sampled portfolio actions in the exploratory formulation is given by
\begin{align}\label{MKVLQ-benckmark-dynamics}
d X^{\Dc, \bm\pi}_s = a^{\bm\pi}_{\delta(s)} \big( b ds + \sigma dW_s + \sigma_o d B_s\big), \quad s \geq t, \quad X^{\Dc, \bm\pi}_t = \xi,
\end{align}
with constant model parameters $b >0, \sigma\geq 0, \sigma_o \geq 0$. The BM $B$ stands for common noise.

We consider a LQ-type performance measure under entropy regularization for the social planner, whose goal is to maximize the expected mean of the population's wealth and meanwhile minimize the systemic risk described by the quadratic distance between the wealth and its conditional mean under common noise. That is, the objective function is given by
\begin{align}\label{MKVLQ-benckmark-reward}
J^\Dc(t, \xi; {\bm \pi}) = \E^e\left[X^{\Dc, \bm\pi}_T - \lambda (X^{\Dc, \bm\pi}_T - \E^e[X^{\Dc, \bm\pi}_T|\Gc_T])^2\big] + \gamma\int_t^T E_{\bm \pi}(\delta(s), X_{\delta(s)}^{\Dc, \bm \pi}, \mu_{\delta(s)}^{\Dc, \bm \pi})ds\right].
\end{align}

This example fits into the framework of LQ-MFC in \cite{Renetal25}. By Theorem 5.1 in \cite{Renetal25}, the optimal policy, the optimal value function and the optimal unregularized Iq-function admit the explicit expressions that
\begin{align*}
{\bm \pi}^*(\cdot|t, x, \mu) &= \Nc\Big(-\frac{b}{\sigma^2 + \sigma_o^2}(x - \bar\mu) + \frac{b}{2\lambda \sigma^2} e^{-\frac{b^2}{\sigma^2 + \sigma_o^2}(t - T)}, \frac{\gamma}{2\lambda(\sigma^2 + \sigma_o^2)} e^{-\frac{b^2}{\sigma^2 + \sigma_o^2}(t - T)}\Big),\\
J^*(t, \mu) &= \bar\mu - \lambda e^{\frac{b^2}{\sigma^2 + \sigma_o^2}(t - T)}{\rm Var}(\mu) +  \frac{\gamma b^2}{4 (\sigma^2 + \sigma_o^2)} (t - T)^2 - \frac{\gamma}{2} \log \frac{\pi \gamma}{(\sigma^2 + \sigma_o^2) \lambda} (t - T)\\
&\;\; + \frac{\sigma^2 + \sigma_o^2}{4\lambda \sigma^2} \Big(e^{-\frac{b^2}{\sigma^2 + \sigma_o^2}(t - T)} - 1 \Big),\\
q^{0, *}(t, \mu, {\bm h})
& = -(\sigma^2 + \sigma_o^2)\lambda e^{\frac{b^2}{\sigma^2 + \sigma_o^2}(t- T)} \int_{\R \times \R}\Big( a + \frac{b}{\sigma^2 + \sigma_o^2} (x -\bar\mu) - \frac{b}{2 \lambda\sigma \sqrt{\sigma^2 + \sigma_o^2}}e^{-\frac{b^2}{\sigma^2 + \sigma_o^2}(t- T)}\\
&\;\; + \big(-1 + \frac{\sigma}{\sqrt{\sigma^2 + \sigma_o^2}}\big)\bar {\bm h}\Big)^2 {\bm h}(a|x)da\mu(dx) +  \frac{\gamma b^2}{2(\sigma^2 + \sigma_o^2)} (t - T) - \frac{\gamma}{2} \log \frac{\pi \gamma}{(\sigma^2 + \sigma_o^2) \lambda},
\end{align*}
where $\bar{\bm h} := \int_{\R \times \R} a {\bm h}(a|x)da \mu(dx)$.


When model parameters are unknown, we can accordingly consider the precise parametrization of the optimal value function $J^*$ and the optimal unregularized Iq-function $q^{0, *}$ by
\begin{align}
J^\theta(t, \mu) &= \bar\mu  - \lambda e^{\theta_1(t - T)} {\rm Var}(\mu) + \frac{\gamma \theta_1}{4}(t - T)^2 + \theta_2(t - T)\label{equ:paraJ} + \theta_3\big(e^{-\theta_1(t - T)} - 1\big),\\
q^{0, \psi}(t, \mu, {\bm h}) & = -\frac{1}{2} e^{\psi_1 + \psi_2(t -T)}\int_{\R \times \R}\Big(a + \psi_3(x - \bar\mu) + \psi_4 e^{-\psi_2(t - T)} \nonumber\\
&  + (\psi_5 -1 )\bar{\bm h}\Big)^2 {\bm h}(a|x)da\mu(dx) - \frac{\gamma}{2} \log{(2\pi \gamma)} - \frac{\gamma}{2}\psi_1 + \frac{\gamma\psi_2}{2}(t - T),\label{equ:q0para}\\
\frac{\delta q^{0, \psi}}{\delta {\bm h}}(t, \mu, {\bm h})(x, a) & = - \frac{1}{2} e^{\psi_1 + \psi_2(t - T)} \Big(a^2 + 2 \psi_3(x - \bar\mu) a  + 2 \psi_4 \psi_5 e^{-\psi_2(t - T)} a - 2(1 - \psi_5^2)\bar{\bm h} a \Big).\label{equ:para-derivative-q}
\end{align}
where $\theta = (\theta_1, \theta_2, \theta_3)\trans \in \R^3$, $\psi = (\psi_1, \ldots, \psi_5)\trans \in \R^5$.
The true values of parameters are
$\theta_1^* = \frac{b^2}{\sigma^2 + \sigma_o^2}, \theta_2^* =  -\frac{\gamma}{2} \log \frac{\pi\gamma}{(\sigma^2 + \sigma_o^2)\lambda}, \theta_3^* = \frac{\sigma^2 + \sigma_o^2}{4\lambda \sigma^2}$,
$\psi_1^*= \log \big(2\lambda(\sigma^2 + \sigma_o^2)\big), \psi_2^* = \frac{b^2}{\sigma^2 + \sigma_o^2}, \psi_3^* = \frac{b}{\sigma^2 + \sigma_o^2}, \psi_4^* = -\frac{b}{2\lambda \sigma \sqrt{\sigma^2 + \sigma_o^2}}, \psi_5^*= \frac{\sigma}{\sqrt{\sigma^2 + \sigma_o^2}}$, respectively. Accordingly, the optimal policy can be parameterized by 
\begin{align}\label{equ-LQ-benchmark-pi}
{\bm \pi}^\psi = \Nc\big(-\psi_3(x- \bar\mu) - \frac{\psi_4}{\psi_5} e^{-\psi_2(t - T)}, \gamma e^{-\psi_1 - \psi_2(t - T)}\big).
\end{align}


To generate data for the simulation experiment, we consider a simulator with access to the conditional mean, conditional variance and aggregated rewards of the social planner with test policies as input. We also choose to parameterize the test policy ${\bm h}$ in the same form of $\bm\pi$ but with the different parameter $\tilde\phi$.

\noindent {\bf The Environment Simulator$_{\Delta t}$}: We first calculate the conditional mean of $X^{\Dc, {\bm h}^{\tilde\phi}}$ and obtain that the conditional mean of the dynamics \eqref{MKVLQ-benckmark-dynamics}, denoted by $\bar\mu_{t}$, satisfies
\begin{align*}
\bar\mu_{t_{k + 1}}\simeq \bar\mu_{t_k} - \tilde\phi_4 e^{-\tilde\phi_2(t - T)} \big(b \Delta t + \sigma_o \sqrt{\Delta t} \Delta W\big),
\end{align*}
where $\Delta t$ denotes the time step and $\Delta W\sim \Nc(0,1)$. Then we calculate the conditional variance of the dynamics \eqref{MKVLQ-benckmark-dynamics} under the policy ${\bm h}^{\tilde\phi}$ denoted as ${\rm Var}(\mu_t): = \E^e\big[(X_t^{\Dc} - \bar\mu_t)^2\big|\Gc_t\big]$. The Euler approximation of ${\rm Var}(\mu_t)$ is given by
\begin{align*}
{\rm Var}(\mu_{t_{k+1}}) & \simeq {\rm Var}(\mu_{t_k}) + \Big(\big(-2b \tilde\phi_3  + \sigma^2 \tilde\phi_3^2 + \sigma_o^2 \tilde\phi_3^2\big){\rm Var}(\mu_{t_k}) + \sigma^2 \tilde\phi_4^2 e^{-2\tilde\phi_2(t_k -T)} \\
& + (\sigma^2 + \sigma_o^2) \gamma e^{-\tilde\phi_1 - \tilde\phi_2(t_k - T)}\Big)\Delta t - 2\sigma_o \tilde\phi_3 {\rm Var}(\mu_{t_k}) \sqrt{\Delta t} \Delta W.
\end{align*}

In what follows, we consider three algorithms for this LQ-MFC example, namely the optimal q-learning algorithm (Algorithm \ref{algo:fixed-point}), the Actor-Critic q-learning algorithm without inner iterations that $L=1$ (Algorithm \ref{algo:actor-critic}-A), and the Actor-Critic q-learning algorithm with inner iterations that $L>1$ (Algorithm \ref{algo:actor-critic}-B). The parametrizations of $J$, $q^0$ and the partial linear functional derivative  of $q^0$ with respect to ${\bm h}$ are given in \eqref{equ:paraJ}-\eqref{equ:para-derivative-q}. In Algorithm \ref{algo:fixed-point}, ${\bm \pi}$ takes the form in \eqref{equ-LQ-benchmark-pi}. In Algorithm \ref{algo:actor-critic}-A and Algorithm \ref{algo:actor-critic}-B, we take advantage of the derived form of ${\bm \pi}$ as in \eqref{equ-LQ-benchmark-pi} but would like to implement the Actor-step to update the policy according to \eqref{inner-iteration}. To this end, we choose to parameterize ${\bm \pi}$ with the parameter $\phi$, which differs from the one in the Iq-function. In addition, by employing the form of the optimal Iq-function in \eqref{equ:q0para}, we can eliminate the consistency loss function in the Actor-step of Algorithm \ref{algo:actor-critic}-A and Algorithm \ref{algo:actor-critic}-B by setting $w_c = 0$. For the simulator, we choose inputs $T =1$, $b=0.25$, $\sigma=0.5$, $\sigma_0 =0.5 $. We also set the model parameters as $\beta =0$, $\gamma=0.5$ and $\lambda =1.5$, the time step $\Delta t = 0.1$, the number of test policies $M = 20$, and set the lower and upper bounds of the uniform distribution for the test policies as $p_i(n)= 0$, $q_i(n)= \frac{1}{n^{0.225}}$, $1 \leq i \leq 5$. We set the initialization of $\bar\mu_0 \sim \Nc([0, 1])$ and  ${\rm Var}(\mu_0) \sim \Uc([0, 1])$ and choose the initialization of $\theta = (-0.5, 0.5, 0.5)\trans$, $\psi = (1, -0.5, 1, -0.5, 0.1)\trans$ in all three algorithms and set $\phi=(1.5, -1, 1.5, -1)\trans$  in Algorithms \ref{algo:actor-critic}-A and \ref{algo:actor-critic}-B, respectively.

In Algorithm \ref{algo:fixed-point}, the number of episodes $N =2000$ and the learning rates are set as
{\small
\begin{align*}
\alpha_\theta &= \big(\frac{0.02}{n^{0.2}}, \frac{0.03}{n^{0.15}}, \frac{0.06}{n^{0.31}}\big),\;\; \alpha_\psi  = \big(\frac{0.03}{n^{0.15}}, \frac{0.09}{n^{0.15}}, \frac{0.02}{n^{0.3}}, \frac{0.015}{n^{0.2}}, \frac{0.09}{n^{0.03}}\big).
\end{align*}}
Numerical results for the learning of $\theta$ and $\psi$ are presented in Figure \ref{figure:Alg1-LQ}.
\begin{figure}[h]
\begin{subfigure}{0.33\textwidth}
\centering
\includegraphics[width=\textwidth]{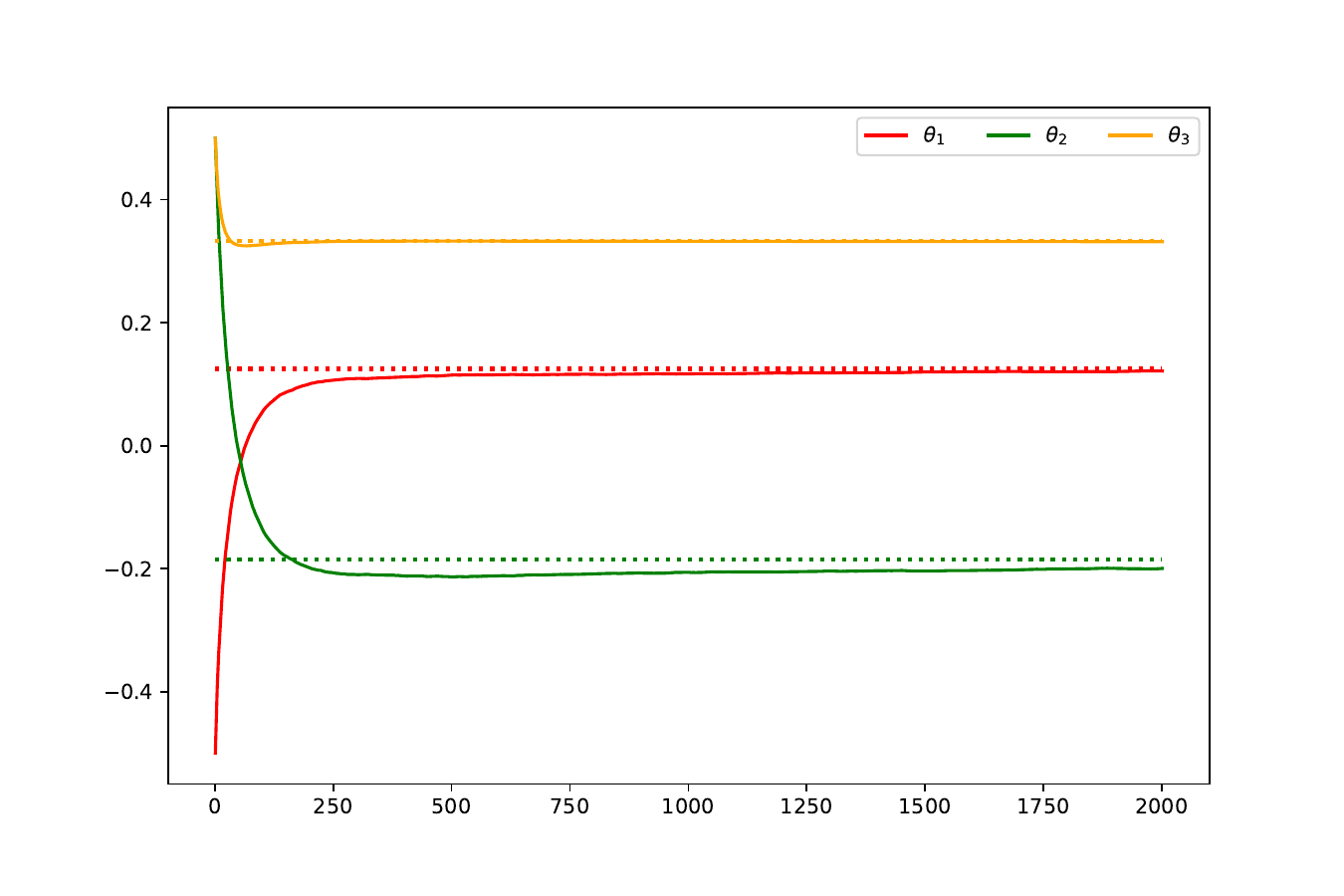}
\caption{Convergence of parameters $\theta$}
\label{LQfigure:Alg1-theta}
\end{subfigure}%
\begin{subfigure}{0.33\textwidth}
\centering
\includegraphics[width=\textwidth]{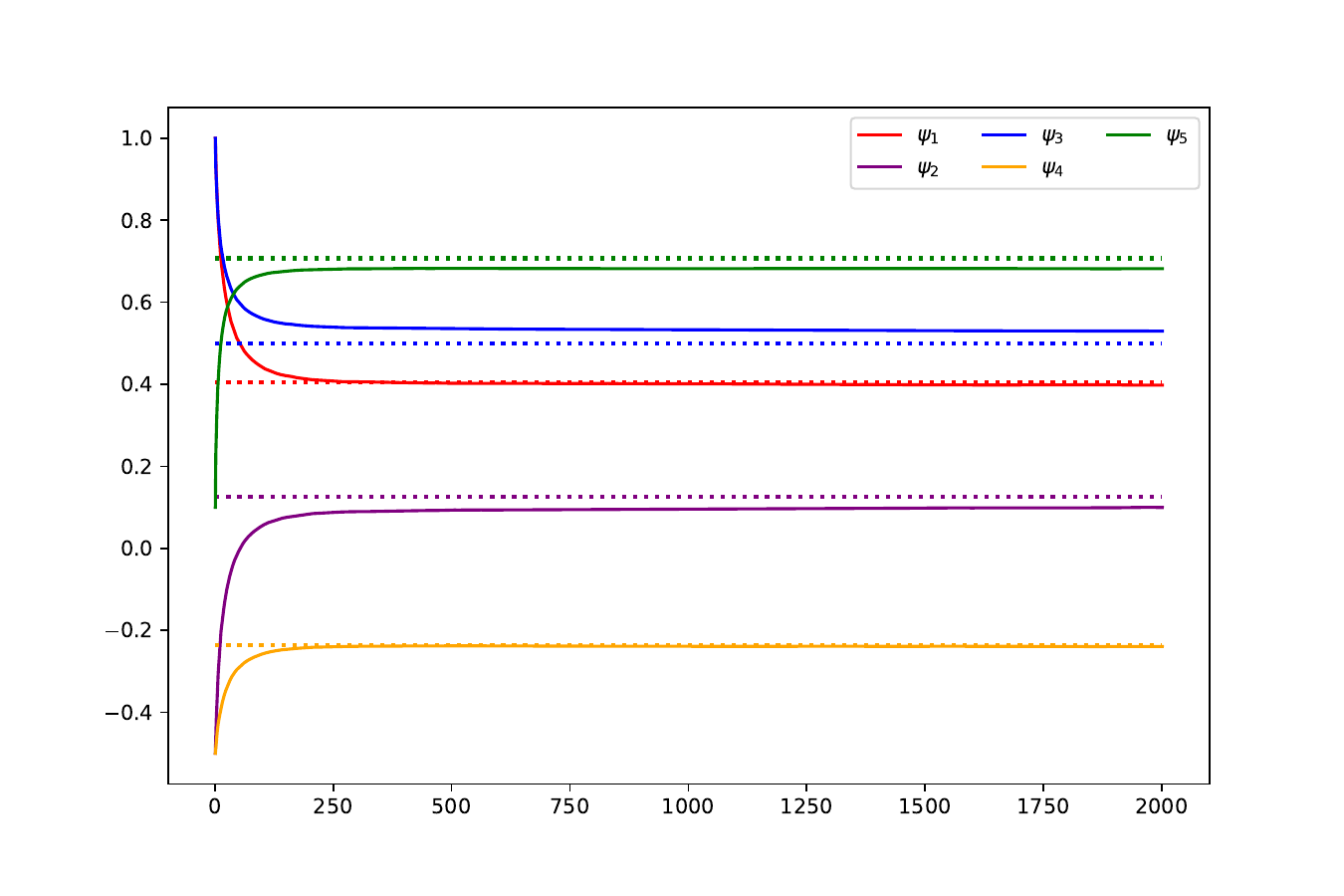}
\caption{Convergence of parameters $\psi$}
\label{LQfigure:Alg1-psi}
\end{subfigure}
\begin{subfigure}{0.33\textwidth}
\centering
\includegraphics[width=\textwidth]{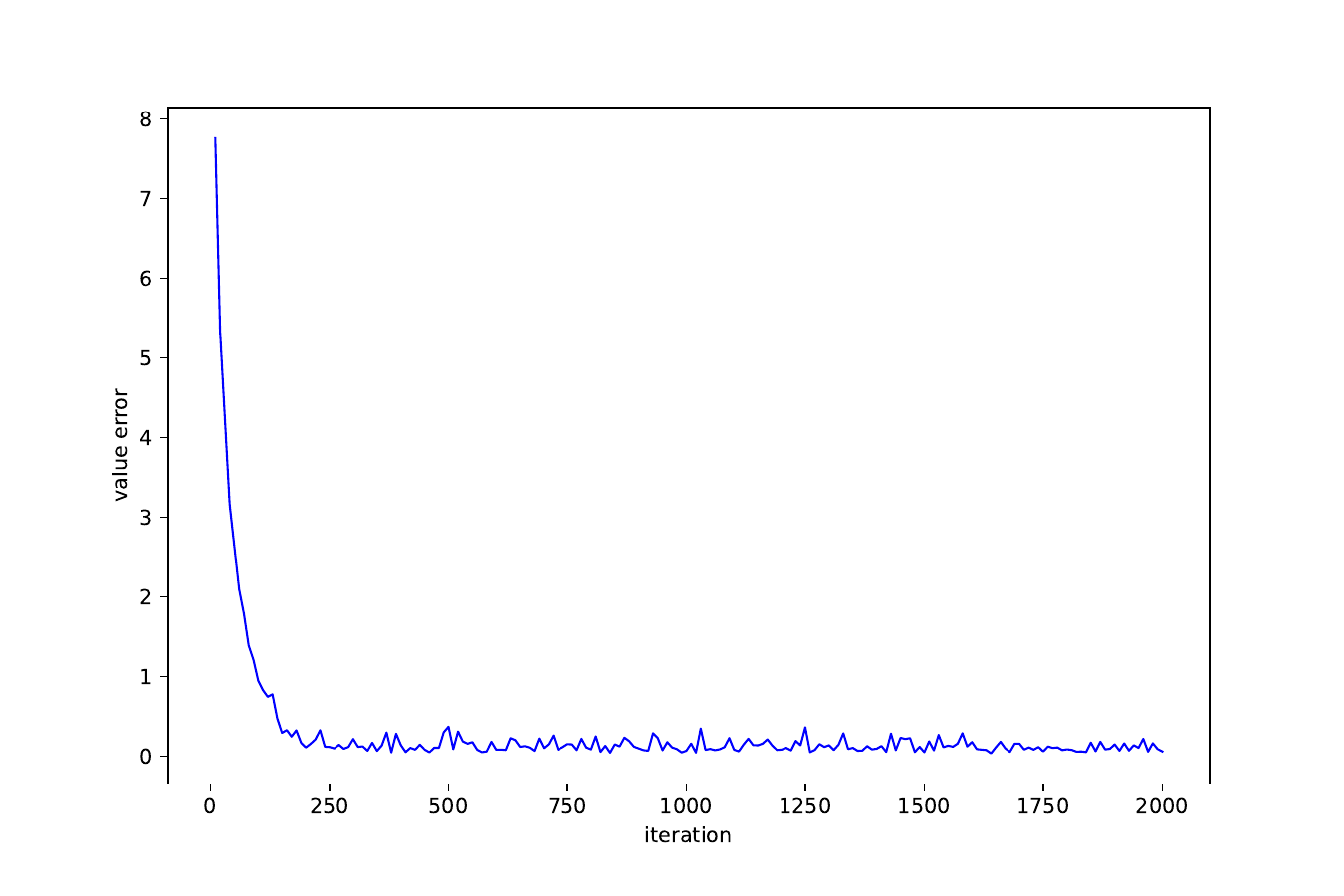}
\caption{$L^1$ error of the value function}
\label{LQfigure:Alg1-J}
\end{subfigure}
\centering
\caption{Convergence of the learnt parameters under Algorithm \ref{algo:fixed-point}} \label{figure:Alg1-LQ}
\end{figure}

In Algorithm \ref{algo:actor-critic}-A, the number of episodes $N=10000$, there is no inner iteration that $L =1$. The learning rates are given by
{\small\begin{align*}
\alpha_\theta  =\big(\frac{0.02}{n^{0.2}}, \frac{0.03}{n^{0.15}}, \frac{0.06}{n^{0.31}}\big), \; \alpha_\psi =\big(\frac{0.04}{n^{0.15}}, \frac{0.12}{n^{0.1}}, \frac{0.02}{n^{0.3}}, \frac{0.015}{n^{0.2}}, \frac{0.25}{n^{0.2}}\big), \; \alpha_\phi = \big(\frac{0.04}{n^{0.15}}, \frac{0.1}{n^{0.1}}, \frac{0.02}{n^{0.3}}, \frac{0.015}{n^{0.2}}\big).
\end{align*}}
Numerical results for the learning of $\theta$ and $\psi$ are presented in Figure \ref{figure:Alg2-LQ}.
\begin{figure}[t!]
\centering
\begin{subfigure}{0.4\textwidth}
\centering
\includegraphics[width=\textwidth]{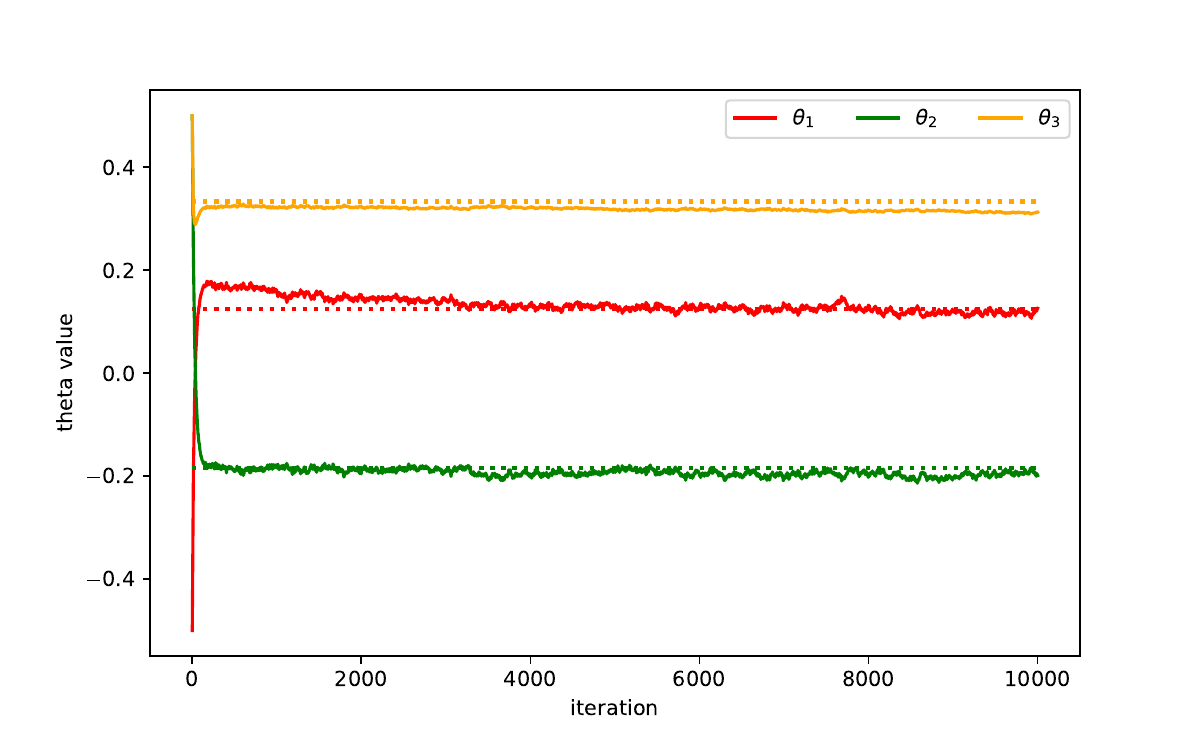}
\caption{Convergence of parameters $\theta$}
\label{LQfigure:Alg2A-theta}
\end{subfigure}%
\begin{subfigure}{0.4\textwidth}
\centering
\includegraphics[width=\textwidth]{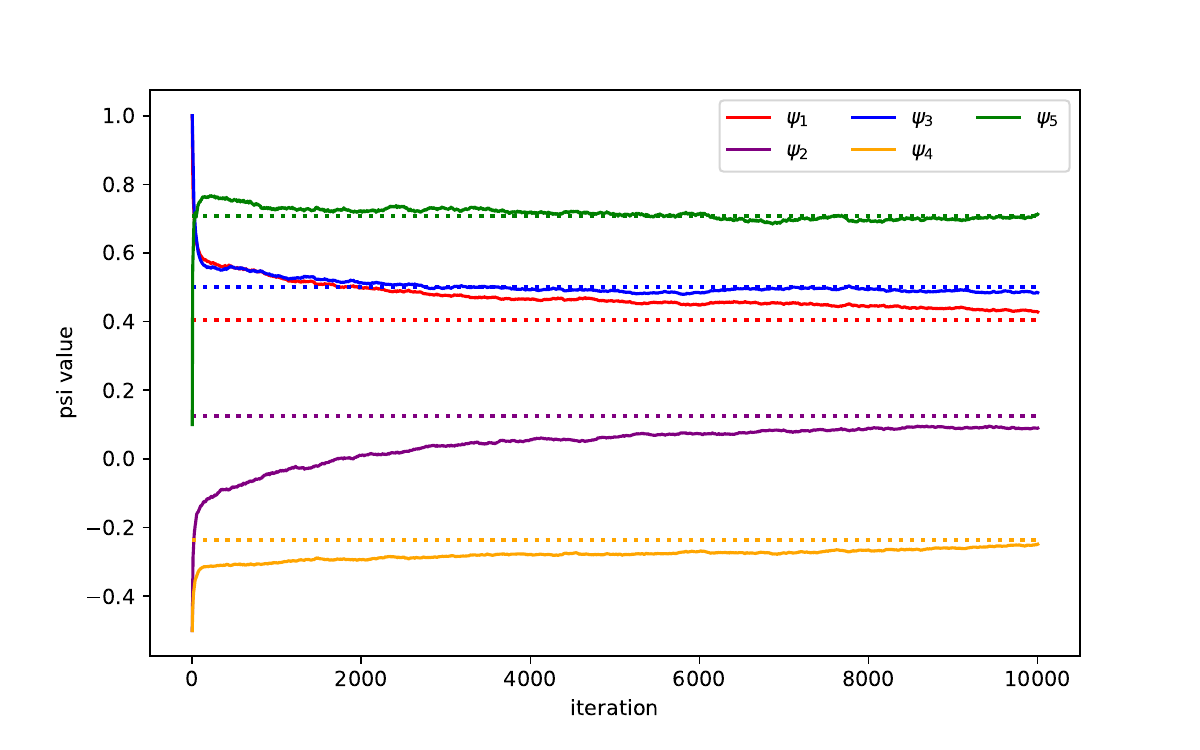}
\caption{Convergence of parameters $\psi$}
\label{LQfigure:Alg2A-psi}
\end{subfigure}
\begin{subfigure}{0.4\textwidth}
\centering
\includegraphics[width=\textwidth]{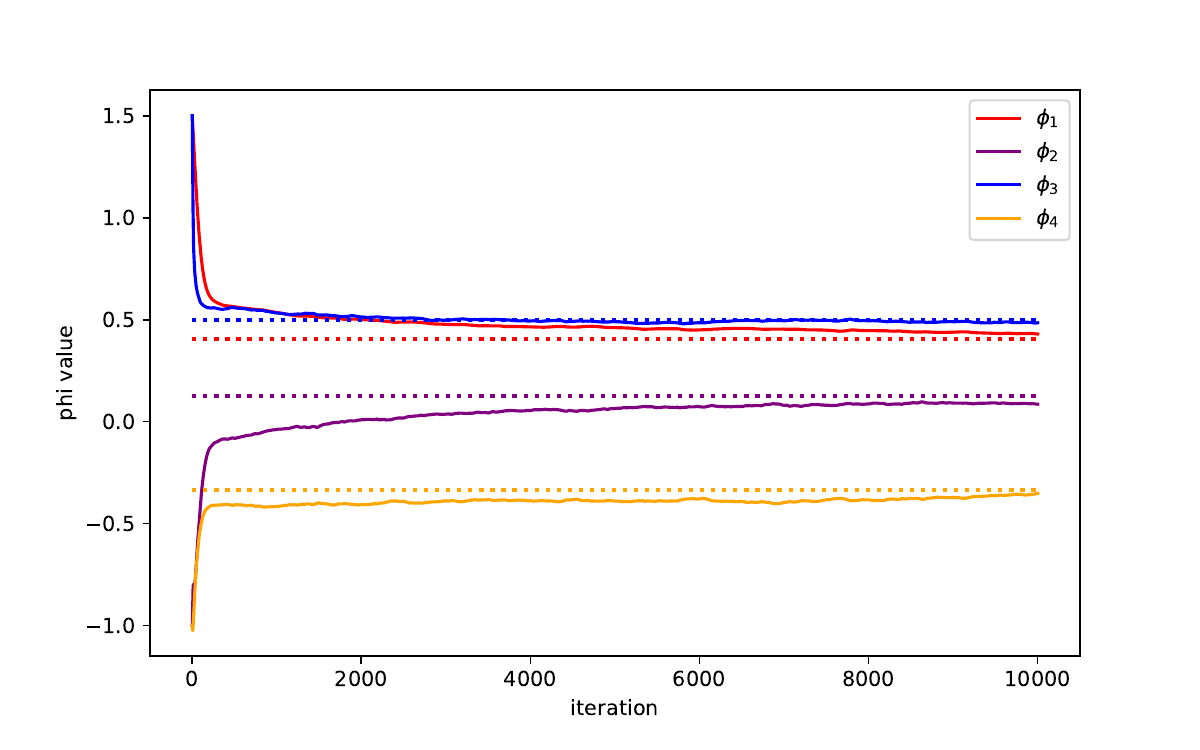}
\caption{Convergence of parameters $\phi$}
\label{LQfigure:Alg2A-value}
\end{subfigure}
\begin{subfigure}{0.4\textwidth}
\centering
\includegraphics[width=\textwidth]{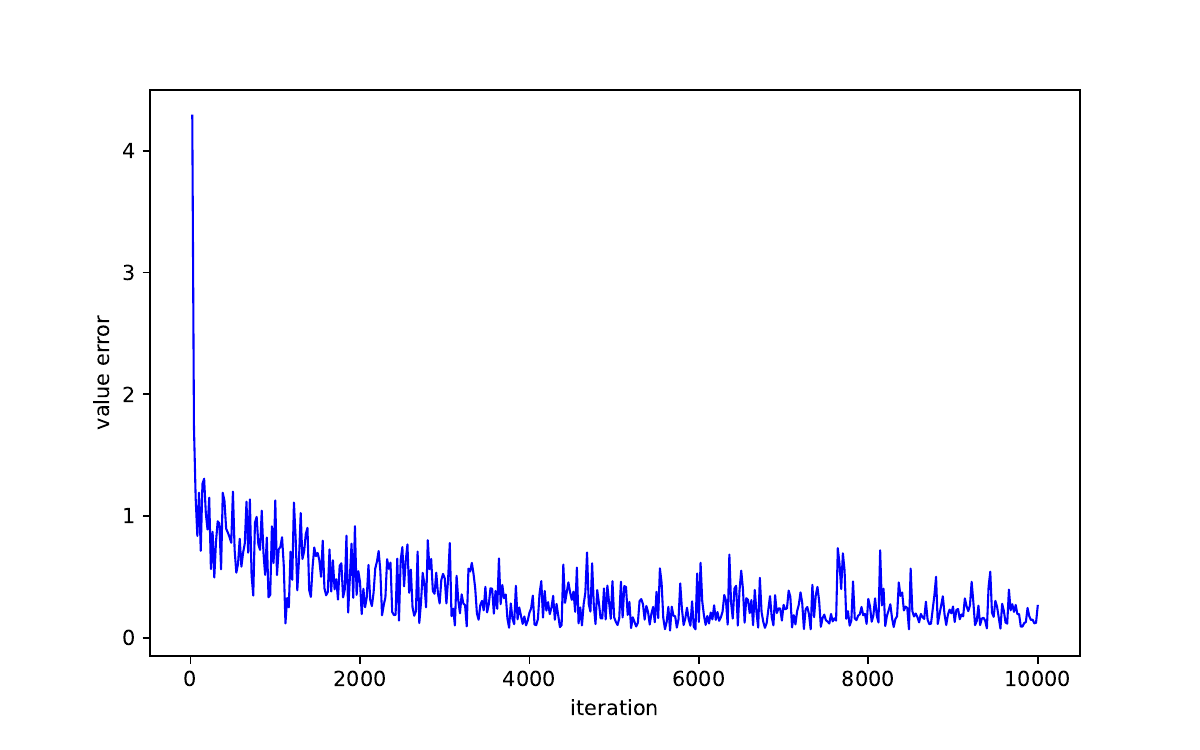}
\caption{$L^1$ error of the value function}
\label{LQfigure:Alg2A-J}
\end{subfigure}
\caption{Convergence of the learnt parameters under Algorithm \ref{algo:actor-critic}-A }\label{figure:Alg2-LQ}
\end{figure}

In Algorithm \ref{algo:actor-critic}-B, we set the number of episodes $N=2000$, the number of inner iterations $L = 25$, and the learning rates that
{\small
\begin{align*}
\alpha_\theta =\big(\frac{0.02}{n^{0.2}}, \frac{0.03}{n^{0.15}}, \frac{0.06}{n^{0.31}}\big), \; \alpha_\psi =\big(\frac{0.03}{n^{0.1}}, \frac{0.09}{n^{0.1}},\frac{0.02}{n^{0.2}},\frac{0.015}{n^{0.2}},\frac{0.15}{n^{0.2}}\big),  \; \alpha_\phi = \big(\frac{0.03}{n^{0.15} l^{0.15}}, \frac{0.08}{n^{0.1} l^{0.1}},\frac{0.02}{n^{0.1} l^{0.1}},\frac{0.015}{n^{0.2}l^{0.2}}\big),
\end{align*}}where $n$ and $l$ are indexes of outer iterations and inner iterations, respectively.

In addition, we evaluate the performance of the learnt policy over iterations by illustrating the value function error. At each iteration, we execute the learnt policy 3000 times to output 3000 trajectories by the simulator and estimate the value function under the learnt policy averaged over these trajectories. The optimal value function is used as a benchmark. We adopt the $L^1$ error between the learnt value function and the true value function along the iterations in three algorithms as shown in Figures \ref{LQfigure:Alg1-J}, \ref{LQfigure:Alg2A-J}, \ref{LQfigure:Alg2B-J}.

\begin{figure}[t!]
\centering
\begin{subfigure}{0.4\textwidth}
\centering
\includegraphics[width=\textwidth]{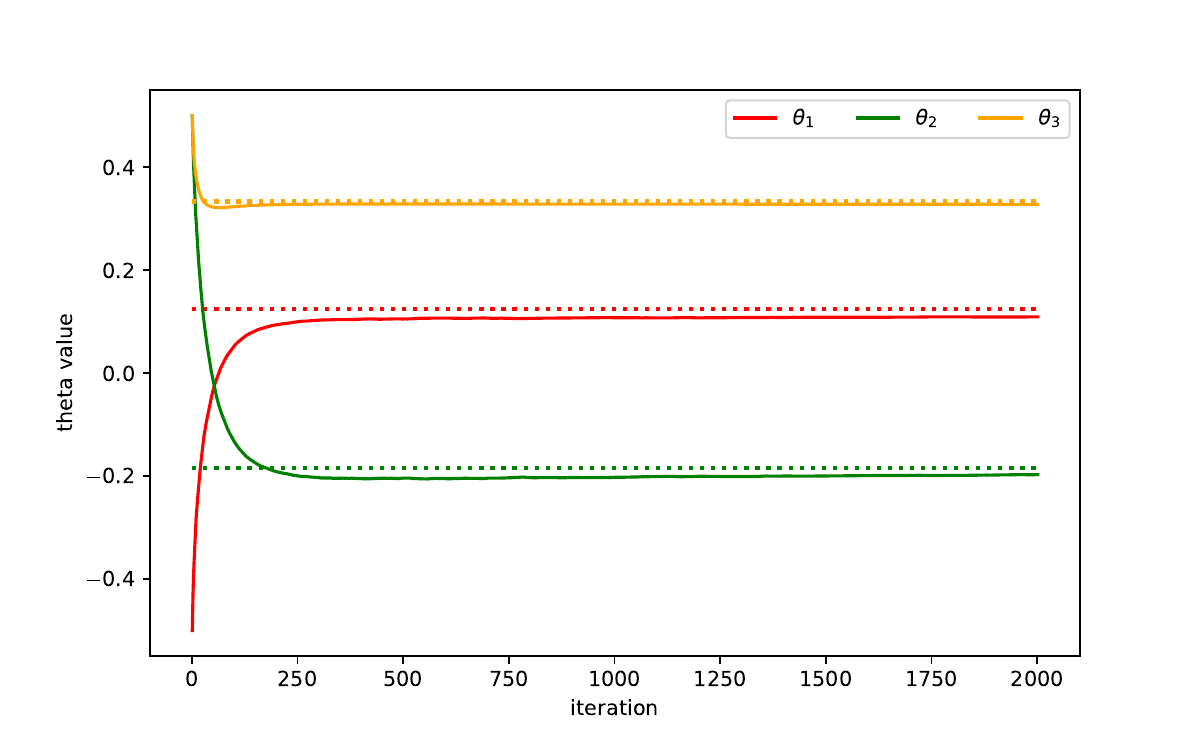}
\caption{Convergence of parameters $\theta$}
\label{LQfigure:Alg2B-theta}
\end{subfigure}%
\begin{subfigure}{0.4\textwidth}
\centering
\includegraphics[width=\textwidth]{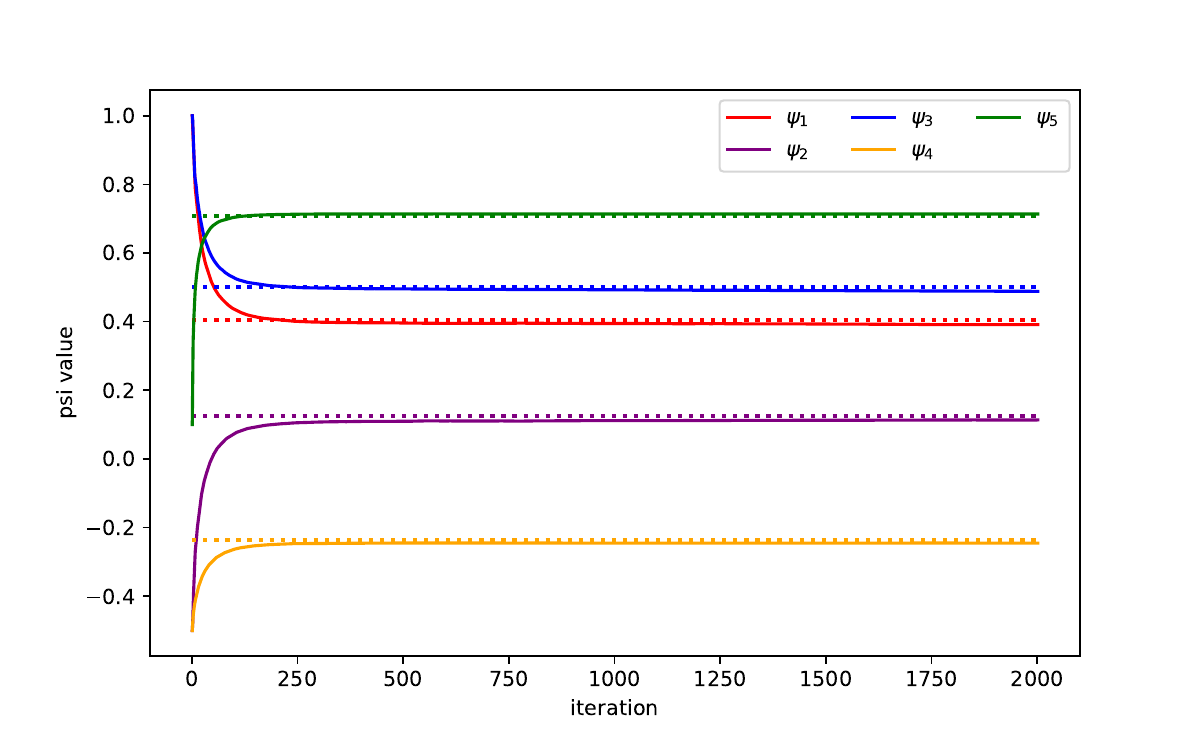}
\caption{Convergence of parameters $\psi$}
\label{LQfigure:Alg2B-psi}
\end{subfigure}
\begin{subfigure}{0.4\textwidth}
\centering
\includegraphics[width=\textwidth]{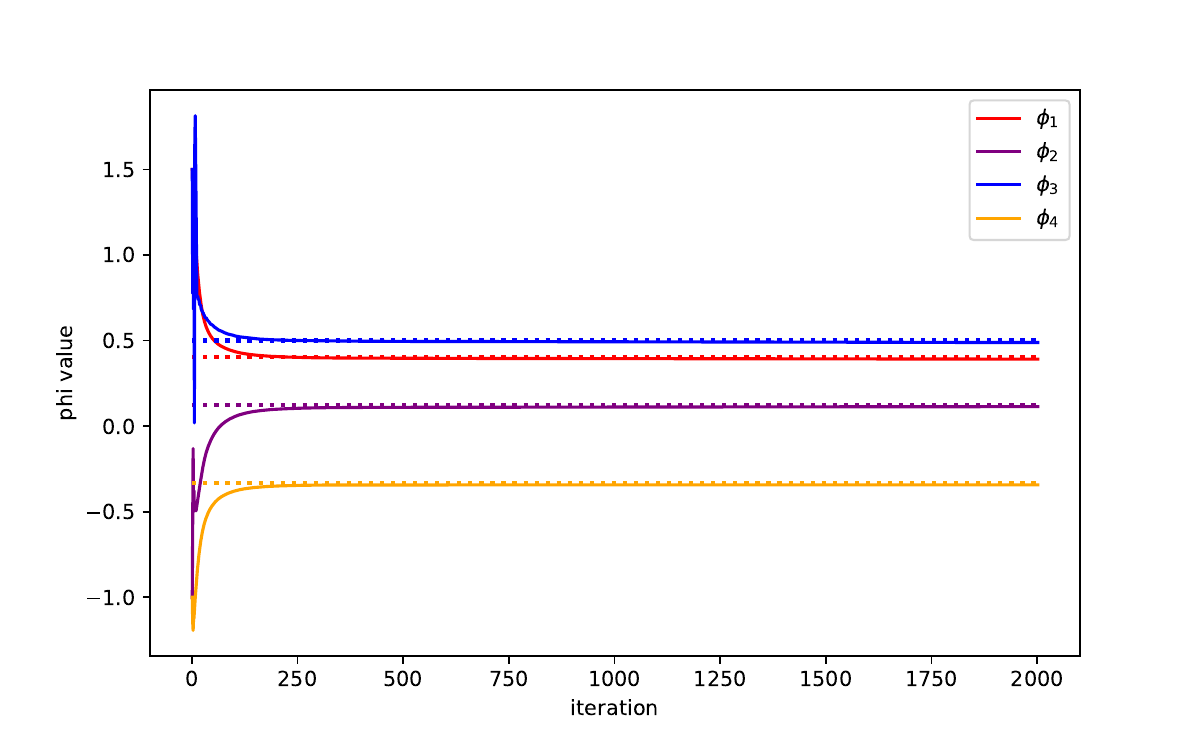}
\caption{Convergence of parameters $\phi$}
\label{LQfigure:Alg2B-value}
\end{subfigure}
\begin{subfigure}{0.4\textwidth}
\centering
\includegraphics[width=\textwidth]{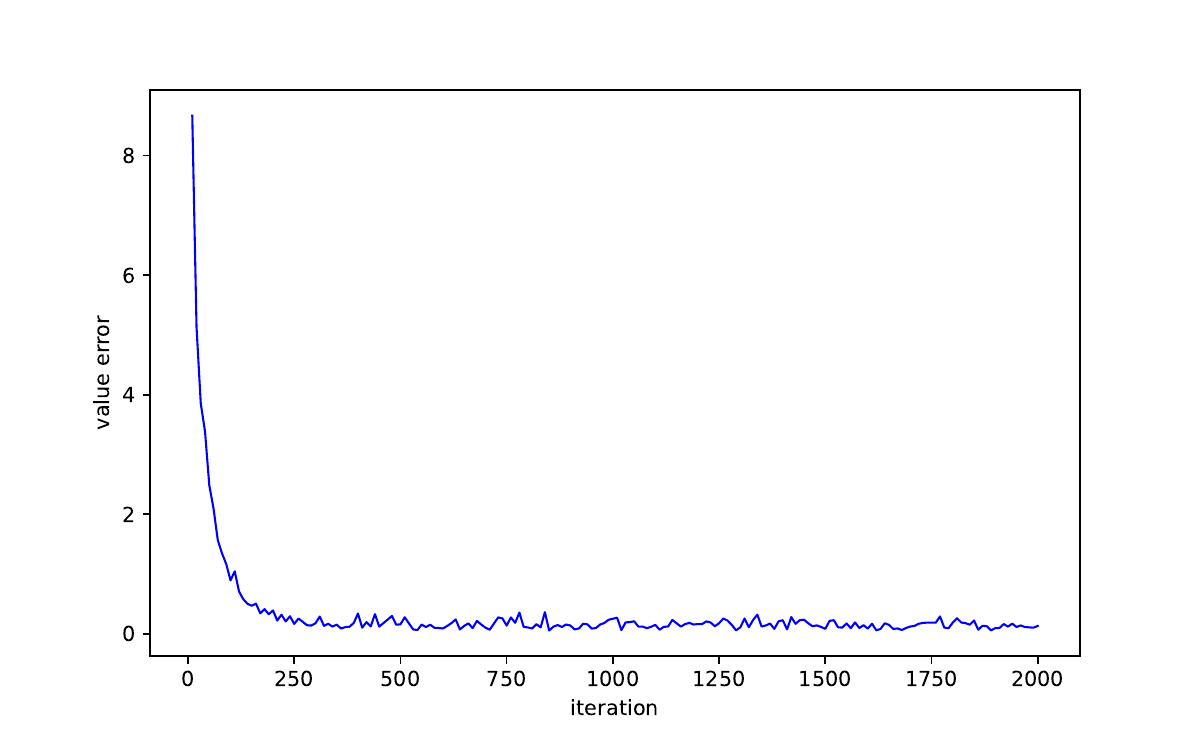}
\caption{$L^1$ error of the value function}
\label{LQfigure:Alg2B-J}
\end{subfigure}
\caption{Convergence of the learnt parameters under Algorithm \ref{algo:actor-critic}-B }\label{figure:Alg3-LQ}
\end{figure}

%

\subsection{Non-LQ-MFC problem}\label{sec:ex-2}
We next consider an example of non-LQ MFC problem where the state process is governed by the conditional McKean-Vlasov SDE under discretely sampled actions that
\begin{align}\label{ex:non-LQ-dynamics}
d X_s^{\Dc, \bm \pi} = a_{\delta(s)}^{\Dc, \bm \pi} \bar\mu_s^{\Dc, \bm \pi}\big(b ds + \sigma dW_s + \sigma_o dB_s\big), \quad s \geq t, \quad X_t^{\Dc, \bm \pi} = \xi \sim \mu.
\end{align}
Here, we assume $\bar\mu = \int_\R x \mu(dx) > 0$, $b>0, \sigma>0$, and $\sigma_o > 0$. 

Let us consider the entropy-regularized logarithmic utility maximization problem on the population mean with the objective function defined by
\begin{align*}
J^{\Dc}(t, \xi; {\bm \pi}) = \E^e\Big[ \gamma \int_t^T e^{-\beta(s -t)}E_{\bm \pi}(\delta(s), X_{\delta(s)}^{\Dc, \bm \pi}, \mu_{\delta(s)}^{\Dc, \bm \pi})ds + e^{-\beta(T - t)}\log \bar\mu_T^{\Dc, \bm \pi} \Big].
\end{align*}

\begin{Lemma} The optimal value function $J^*$ and the optimal policy ${\bm \pi}^*$ admit the form that
\begin{align*}
J^*(t, \mu) &= e^{\beta(t - T)} \log \bar\mu + C(t),\\
{\bm \pi}^*(a|t, x, \mu) & = \Exp\Big(\frac{\sigma_o^2}{\gamma} e^{\beta(t - T)} \Big(\sqrt{(\frac{b}{2\sigma_o^2})^2 + \frac{\gamma}{\sigma_o^2} e^{-\beta(t - T)}} - \frac{b}{2\sigma_o^2}\Big)\Big),
\end{align*}
where {\small
\begin{align}
C(t) &= A_1 e^{\beta(t - T)}  + A_2(T - t) e^{\beta(t - T)} + A_3 e^{\beta(t - T)} \log \big(2\sqrt{e^{2\beta(t - T)} +  \frac{4\gamma \sigma_o^2}{b^2} e^{\beta(t -T)} } + 2 e^{\beta(t - T)} + \frac{4\gamma \sigma_o^2}{b^2}\big)  \nonumber\\
&+ A_4 e^{\beta(t - T)}\sqrt{ 1+  \frac{4\gamma \sigma_o^2 }{b^2} e^{-\beta(t -T)}}+ A_5 \log \Big(\sqrt{1 + \frac{4\gamma\sigma_o^2}{ b^2} e^{-\beta(t - T)}} + 1\Big) + A_6,\label{non-LQ:C(t)}
\end{align}}
and the constants $A_1$ to $A_6$ are given respectively by
\begin{align*}
A_1 &= \frac{b^2}{4 \beta\sigma_o^2}\log \Big(2\sqrt{1 +  \frac{4\gamma \sigma_o^2}{b^2} } + 2 + \frac{4\gamma \sigma_o^2}{b^2}\Big) + \frac{\gamma}{\beta} \log (\frac{b}{2\sigma_o^2}) - \frac{b^2}{4\beta\sigma_o^2}\sqrt{1 + \frac{4\gamma \sigma_o^2}{b^2}},  \\
  &- \frac{\gamma}{\beta} + \frac{\gamma}{\beta}\log\big(\sqrt{1 + \frac{4\gamma\sigma_o^2}{ b^2}} + 1\big),\quad A_2  = \frac{b^2}{4\sigma_o^2},\quad A_3 =-\frac{b^2}{4 \beta \sigma_o^2},\\
A_4 &= \frac{b^2 }{4\beta \sigma_o^2}, \quad A_5 = -\frac{\gamma}{\beta}, \quad A_6 = \frac{\gamma}{\beta} - \frac{\gamma}{\beta} \log \big(\frac{b}{2\sigma_o^2}\big).
\end{align*}
\end{Lemma}

\begin{proof}
We first conjecture that $J^*$ takes the form of $J^*(t, \mu) = A(t) \log \bar\mu + C(t)$.
It follows that
\begin{align*}
\partial_\mu J^*(t, \mu)(x) = \frac{A(t)}{\bar\mu}, \; \partial_x\partial_\mu J^*(t, \mu)(x) =0, \; \partial_\mu^2 J^*(t, \mu)(x, x') = -\frac{A(t)}{\bar\mu^2},
\end{align*}
and $\mathscr{H}(t, \mu, {\bm h}; {\bm \pi}^*) = A(t) \big(b \bar{\bm h} - \frac{1}{2} \sigma_o^2 \bar{\bm h}^2\big)$. By the two-layer fixed point condition in \eqref{twofix}, we get 
\begin{align}\label{equ:first-order-nonLQ}
A(t) b a  - a \sigma_o^2 A(t) \int_{\R \times \R} a'{\bm h}^*(a'|t, x', \mu)da' \mu(dx') - \gamma \log {\bm h}^*(a|t, x, \mu) = \kappa(t, x, \mu).
\end{align}
To derive the explicit optimal policy as the two-layer fixed point, we consider ${\bm h}^*$ as an exponential distribution with the parameter $\lambda(t, x, \mu)$ that
\begin{align}\label{equ:h-normal}
-\gamma \log {\bm h}^*(\cdot|t, x, \mu) = \gamma \lambda(t, x, \mu) a - \gamma \log \lambda(t, x, \mu).
\end{align}
Plugging \eqref{equ:h-normal} into \eqref{equ:first-order-nonLQ} and noting that $\lambda(t, x, \mu)>0$ is independent of $x$ and $\mu$, we get that
\begin{align}\label{equ:mean-variance-h}
\frac{1}{\lambda(t)} =\sqrt{\big(\frac{b}{2\sigma_o^2}\big)^2 + \frac{\gamma}{A(t)\sigma_o^2}} + \frac{b}{2\sigma_o^2} > 0.
\end{align}
Substituting \eqref{equ:mean-variance-h} into \eqref{equ:dynamic-programming-equation}, we have that
\begin{align*}
& \Big(A'(t) - \beta A(t)\Big) \log \bar\mu
+ C'(t) -\beta C(t) + \frac{1}{2} A(t) b  \Big(\sqrt{\big(\frac{b}{2\sigma_o^2}\big)^2 + \frac{\gamma}{A(t)\sigma_o^2}} + \frac{b}{2\sigma_o^2}\Big)+ \frac{\gamma}{2}\\
& -\gamma \log \Big(\sqrt{\big(\frac{b}{2\sigma_o^2}\big)^2 + \frac{\gamma}{A(t)\sigma_o^2}} + \frac{b}{2\sigma_o^2}\Big)= 0.
\end{align*}
Together with $A(T) =1$, we first get that $A(t)  =  e^{\beta(t - T)} >0$.
We then obtain the ODE of $C(t)$ with the terminal condition $C(T) = 0$ that
\begin{align*}
& C'(t) -\beta C(t) + \frac{b^2}{4 \sigma_o^2}e^{\beta(t - T)} + \frac{\gamma}{2}
 + \frac{1}{2} b \sqrt{\big(\frac{b}{2\sigma_o^2}\big)^2 e^{2\beta(t -T)} + \frac{\gamma}{\sigma_o^2}e^{\beta(t -T)}} \\
&- \gamma \log \Big(\sqrt{\big(\frac{b}{2\sigma_o^2}\big)^2
 + \frac{\gamma}{\sigma_o^2 e^{\beta(t - T)}}} + \frac{b}{2\sigma_o^2}\Big) = 0,
\end{align*}
for which we can obtain the explicit solution as in \eqref{non-LQ:C(t)}.

For the well-posedness of the problem, let us also verify that $\bar\mu_s >0$ for any $s \in [t, T]$ under the optimal policy ${\bm \pi}^* = \Exp(\lambda(t))$ whenever $\bar\mu > 0$. Recall that $a_{\delta(r)} \sim \Exp(\lambda(\delta(r)))$, and thus $\E^e[a_{\delta(r)}] = \frac{1}{\lambda(\delta(r))}$. We take the conditional expectation $\E^e[\cdot|\Gc_s]$  on both sides of \eqref{ex:non-LQ-dynamics} and deduce that, for $\P^e$-a.s. $\omega_0$,
\begin{align}
 \bar\mu_s 
 = &\bar\mu + \E^e\Big[\int_t^s \frac{1}{\lambda(\delta(r))} \Big(b \bar\mu_r(\omega^0)dr + \sigma \bar\mu_r(\omega^0) dW_r + \sigma_o \bar\mu_r (\omega^0)dB_r(\omega^0)\Big)\big|\Gc_s\Big] \nonumber\\
 = & \bar\mu +  \int_t^s \frac{1}{\lambda(\delta(r))}\Big(b \bar\mu_r(\omega^0)dr + \sigma_o \bar\mu_r (\omega^0)dB_r(\omega^0)\Big), \label{nonLQ:conditional-mean}
\end{align}
where in the last equality we have used $\E^e\big[\int_t^s \sigma \bar\mu_r dW_r\big|\Gc_s\big] = 0$ by some standard localization arguments and martingale property. Therefore, $\bar\mu_s$ is a geometric Brownian motion, which implies that $\bar\mu_s> 0$ whenever  $\bar\mu >0$.
\end{proof}

When model parameters $b$, $\sigma$, $\sigma_o$ are unknown, we can consider the parameterization of $J^*$ and $q^{0, *}$ based on the above explicit expressions that
\begin{align}\label{equ:paraJ-nonLQ}
J^\theta (t, \mu) & =  e^{\beta(t - T)} \log \bar\mu + \theta_1 e^{\beta(t - T)} - \theta_1  + (T - t) e^{ \theta_2  + \beta(t - T)}  \\
& - \frac{1}{\beta} e^{\theta_2 + \beta(t - T)} \log \Big(2\sqrt{e^{2\beta(t - T)} + \gamma e^{-\theta_2 + \beta(t- T)}} + 2 e^{\beta(t - T)} + \gamma e^{-\theta_2}\Big) \nonumber\\
& + \frac{1}{\beta} e^{\theta_2}\log \big(2\sqrt{1 + \gamma e^{-\theta_2}} + 2 + \gamma e^{-\theta_2}\big)  \nonumber\\
& - \frac{\gamma}{\beta}\log \big(\sqrt{1 +  \gamma e^{-\theta_2 -\beta(t - T)}} + 1\big) + \frac{\gamma}{\beta}\log\big(\sqrt{1 + \gamma e^{-\theta_2}} + 1\big) \nonumber\\
& + \frac{1}{\beta}e^{\theta_2 + \beta(t - T)}\sqrt{1 +  \gamma e^{-\theta_2 -\beta(t - T)}} - \frac{1}{\beta} e^{\theta_2} \sqrt{1 + \gamma e^{-\theta_2}}, \nonumber\\
q^{0, \psi}(t, \mu, {\bm h}) & = \frac{1}{2}e^{\psi_1 + \beta(t - T)} \big(4 e^{\psi_2} \bar{\bm h} - \bar{\bm h}^2\big) - \psi_3 e^{\beta(t - T)}   - \psi_3 e^{\beta(t - T)}\sqrt{1 + \gamma e^{-\psi_3 -\beta(t - T)}} , \nonumber\\
&\;+ \gamma \log \big(\sqrt{1 +  \gamma e^{-\psi_3 -\beta(t - T)}}  + 1\big) + \gamma \psi_2- \frac{\gamma}{2}, \\
\frac{\delta q^{0, \psi}}{\delta {\bm h}}(t, \mu, {\bm h})(x, a) & = \frac{1}{2}e^{\psi_1 + \beta(t - T)}(4 e^{\psi_2} - 2 \bar {\bm h})a,\label{equ:para-derivative-q-nonLQ}
\end{align}
where $\theta = (\theta_1, \theta_2) \in \R^2$ and $\psi= (\psi_1, \psi_2, \psi_3) \in \R^3$. Then, the true values of parameters are
$\theta_1^* = \frac{b^2}{4 \beta\sigma_o^2}\log \Big(2\sqrt{1 +  \frac{4\gamma \sigma_o^2}{b^2} } + 2 + \frac{4\gamma \sigma_o^2}{b^2}\Big) + \frac{\gamma}{\beta} \log (\frac{b}{2\sigma_o^2}) - \frac{b^2}{4\beta\sigma_o^2}\sqrt{1 + \frac{4\gamma \sigma_o^2}{b^2}}$,
$\theta_2^* = \log\big(\frac{b^2}{4\sigma_o^2}\big)$,
and $\psi_1^* =\log (\sigma_o^2), \psi_2^* = \log (\frac{b}{2\sigma_o^2}), \psi_3^* = \log(\frac{b^2}{4 \sigma_o^2})$, respectively. Accordingly, the optimal policy can be parameterized by 
\begin{align}\label{equ-NLQ-benchmark-pi}
{\bm \pi}^{\psi} = \Exp\Big(\frac{1}{\gamma} e^{\psi_1 + \beta(t - T)}\big(-e^{\psi_2} + \sqrt{e^{2\psi_2} + \gamma e^{-\psi_1 - \beta(t -T)}}\big)\Big).
\end{align}

\noindent {\bf The Environment Simulator $_{\Delta t}$}:  For the simulator, we need to compute the conditional log mean of $X^{\Dc}$ under the policy ${\bm h}^{\tilde\phi} = \Exp\Big(\frac{1}{\gamma} e^{\tilde \phi_1 + \beta(t - T)}\big(-e^{\tilde\phi_2} + \sqrt{e^{2\tilde\phi_2} + \gamma e^{-\tilde\phi_1 - \beta(t -T)}}\big)\Big)$. From \eqref{nonLQ:conditional-mean}, we get the evolution of the conditional log mean denoted by $\log\bar\mu_{t}$ that
\begin{align*}
\log \bar\mu_{t_{k + 1}} & \simeq \Big( b\big(e^{\tilde\phi_2} + \sqrt{e^{2 \tilde\phi_2} + \gamma e^{-\tilde\phi_1 - \beta(t_k - T)}} - \frac{1}{2} \sigma_o^2  \big(e^{\tilde\phi_2} + \sqrt{e^{2 \tilde\phi_2} + \gamma e^{-\tilde\phi_1 - \beta(t_k - T)}}\big)^2\Big)\Delta t\\
& \;\;\; + \sigma_o \big(e^{\tilde\phi_2} + \sqrt{e^{2 \tilde\phi_2} + \gamma e^{-\tilde\phi_1 - \beta(t_k - T)}}\big) \sqrt{\Delta t} \Delta W,
\end{align*}
where $\Delta t$ denotes the time step and $\Delta W\sim \Nc(0,1)$.

Thanks to the exact parameterization of targeted functions in this non-LQ-MFC example, the implementation of the optimal q-learning algorithm (Algorithm \ref{algo:fixed-point}) is again straightforward and hence omitted. On the other hand, it is still an open problem whether the inner iterations in the Actor-step converge or not in a non-LQ setting. Therefore, we shall implement and present in this example both the Actor-Critic q-learning algorithm without inner iterations that $L=1$ (Algorithm \ref{algo:actor-critic}-A) and the Actor-Critic q-learning algorithm with inner iterations that $L>1$ (Algorithm \ref{algo:actor-critic}-B). In particular, we would like to numerically confirm the convergence of inner iterations in this example.

In Algorithm \ref{algo:actor-critic}-A and Algorithm \ref{algo:actor-critic}-B, ${\bm \pi}$ takes the similar form of \eqref{equ-NLQ-benchmark-pi} but again with the parameter $\phi$ differing from the parameter $\psi$ in $q^{0,\psi}$. For the simulator, we choose inputs $T =1$, $b=1.5$, $\sigma=0.5$, $\sigma_o =1$. Moreover, we set the model parameters as $\beta =1$, $\gamma=0.2$, the time step size $\Delta t = 0.05$, the number of test policies $M = 10$ and set the lower and upper bounds of the uniform distribution for the test policies as $p_i(n) = 0$, $q_i(n) = \frac{0.4}{n^{0.2}}$, $1 \leq i \leq 3$. We set the initialization of $\log\bar\mu_0 \sim \Nc([0, 1])$, and choose the initialization of $\theta = (-0.5, 0.5)$, $\psi = (0.5, 0.5, 0.5)$ and $\phi=(1.5, -1, 1.5, -1)$ in Algorithms \ref{algo:actor-critic}-A and \ref{algo:actor-critic}-B, respectively.

In Algorithm \ref{algo:actor-critic}-A, there is no inner iteration that $L = 1$, and we set the number of episodes $N=10000$ and the learning rates as
{\small
\begin{align*}
&\alpha_\theta = \left\{
\begin{aligned}
&\left(\frac{0.01}{n^{0.49}}, \frac{0.01}{n^{0.49}}\right), \; & \mbox {if}\; 1 \leq n \leq 2000,\\
&\left(\frac{0.01}{4 \times n^{0.49}}, \frac{0.01}{3 \times n^{0.49}}\right), \; & \mbox {otherwise},
\end{aligned}\right.\quad
\alpha_\psi = \left\{
\begin{aligned}
& \left(\frac{0.002}{n^{0.51}}, \frac{0.002}{n^{0.57}}, \frac{0.01}{n^{0.32}}\right), \; & \mbox {if}\; 1 \leq n \leq 2000,\\
&\left(\frac{0.002}{n^{0.51}}, \frac{0.002}{n^{0.69}}, \frac{0.01}{n^{0.64}}\right), \; & \mbox {otherwise}.
\end{aligned}
\right.\\
& \alpha_\phi  = \left(\frac{0.06}{n^{0.04}}, \frac{0.02}{n^{0.2}}\right).
\end{align*}}
Numerical results for the learning of $\theta$ and $\psi$ are presented in Figure \ref{figure:Alg2-NLQ}.
\begin{figure}[t!]
\centering
\begin{subfigure}{0.4\textwidth}
\centering
\includegraphics[width=\textwidth]{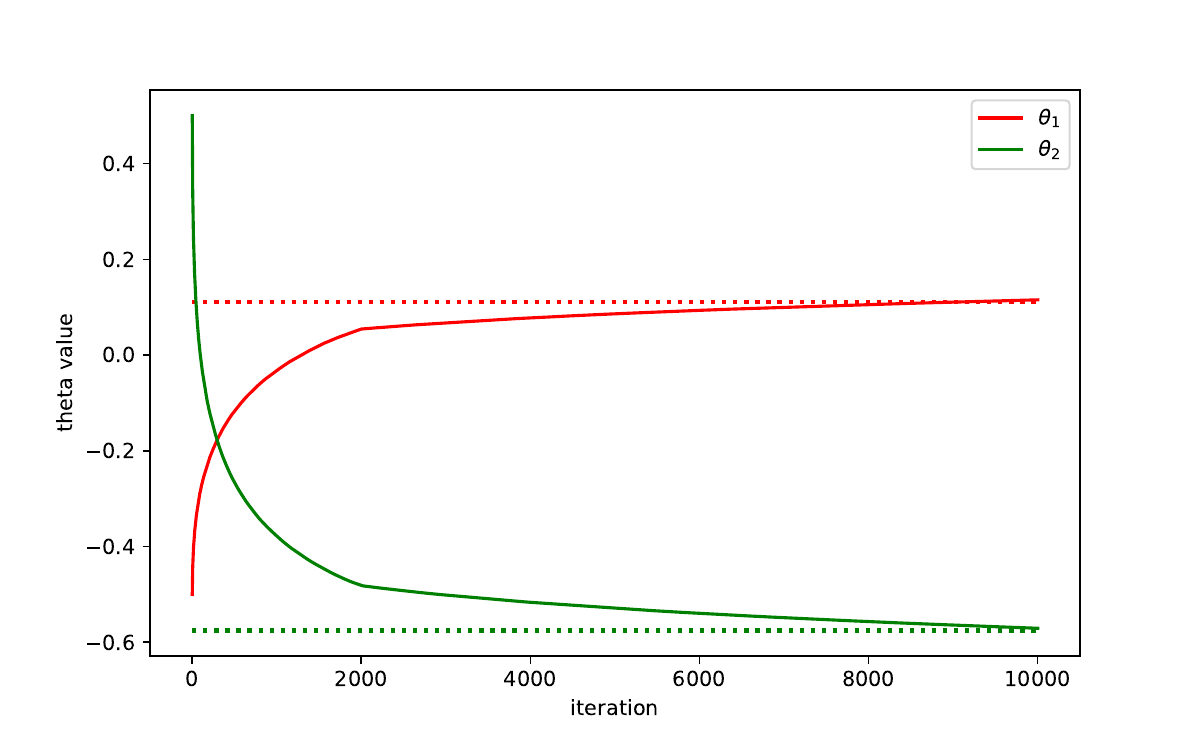}
\caption{Convergence of parameters $\theta$}
\label{NLQfigure:Alg2A-theta}
\end{subfigure}%
\begin{subfigure}{0.4\textwidth}
\centering
\includegraphics[width=\textwidth]{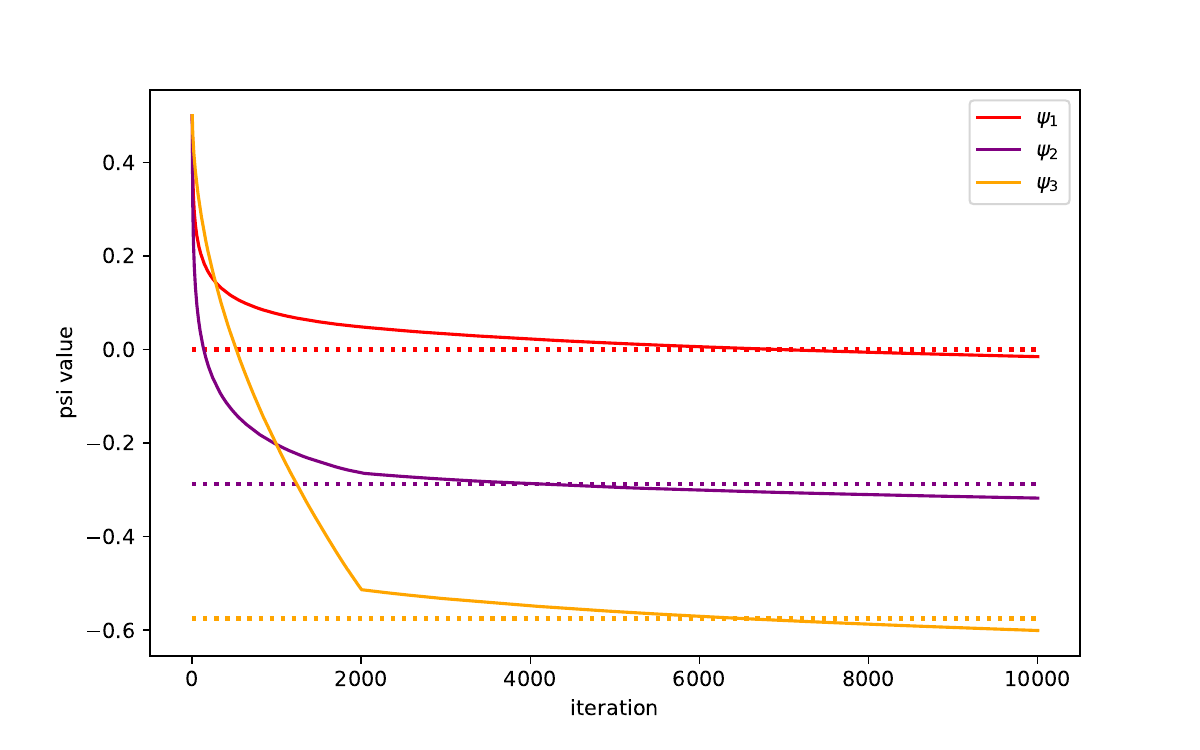}
\caption{Convergence of parameters $\psi$}
\label{NLQfigure:Alg2A-psi}
\end{subfigure}
\begin{subfigure}{0.4\textwidth}
\centering
\includegraphics[width=\textwidth]{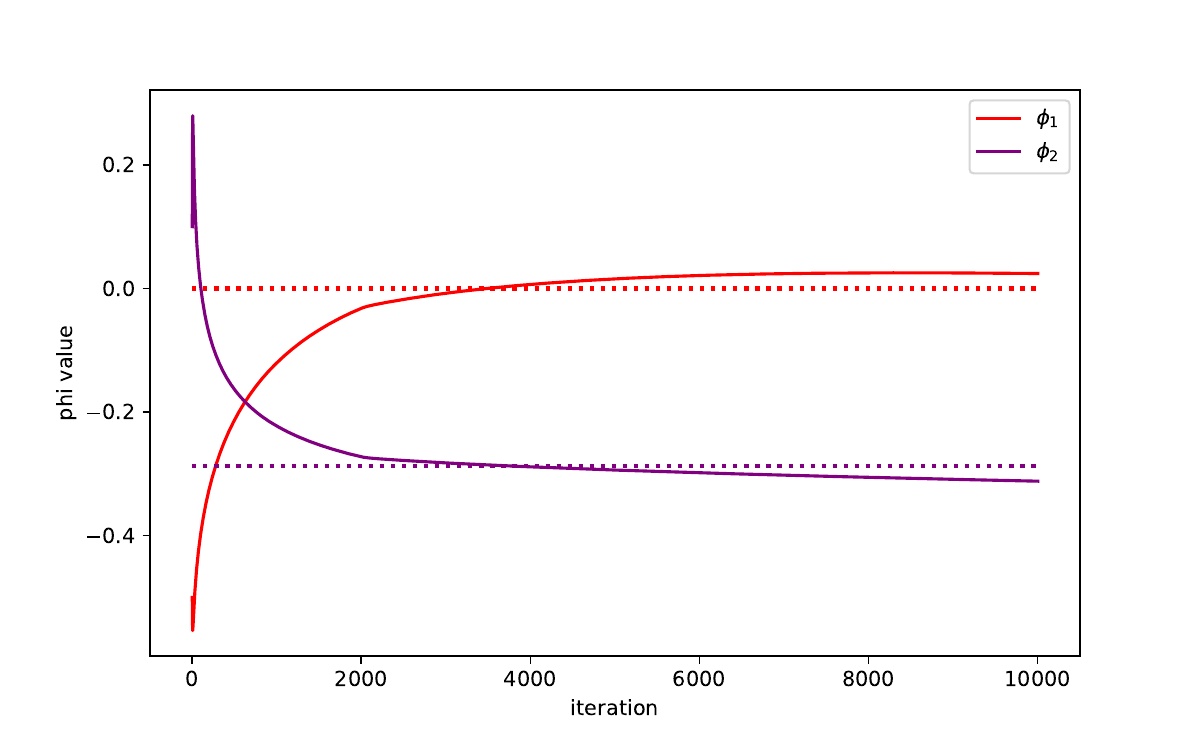}
\caption{Convergence of parameters $\phi$}
\label{NLQfigure:Alg2A-value}
\end{subfigure}
\begin{subfigure}{0.4\textwidth}
\centering
\includegraphics[width=\textwidth]{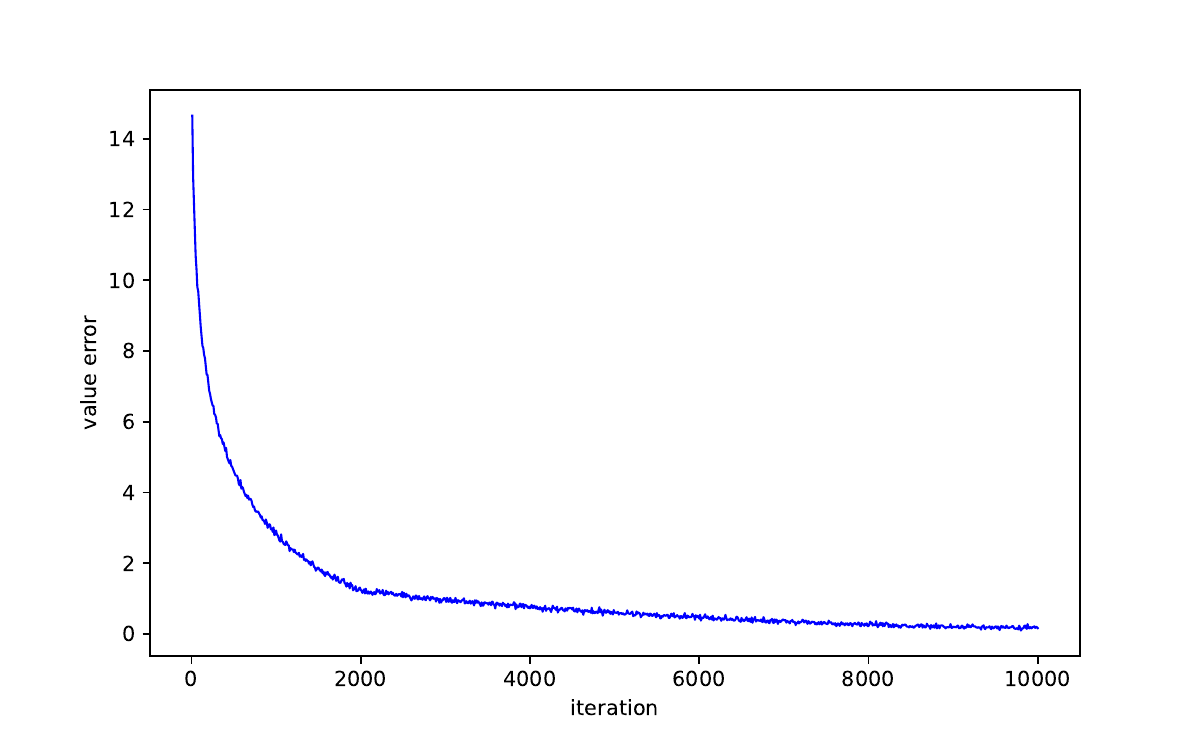}
\caption{$L^1$ error of the value function}
\label{NLQfigure:Alg2A-NLQ-J-error}
\end{subfigure}
\caption{Convergence of the learnt parameters under Algorithm \ref{algo:actor-critic}-A }\label{figure:Alg2-NLQ}
\end{figure}

In Algorithm \ref{algo:actor-critic}-B, we choose the number of inner iterations $L = 30$, and we set the number of episodes $N =5000$ the learning rates as
{\small
\begin{align*}
&\alpha_\theta = \left\{
\begin{aligned}
&\left(\frac{0.01}{n^{0.48}}, \frac{0.01}{n^{0.47}}\right), \; & \mbox {if}\; 1 \leq n \leq 2000,\\
&\left(\frac{0.01}{4 \times n^{0.49}}, \frac{0.01}{3 \times n^{0.49}}\right), \; & \mbox {otherwise},
\end{aligned}\right.\quad
\alpha_\psi = \left\{
\begin{aligned}
& \left(\frac{0.002}{n^{0.51}}, \frac{0.002}{n^{0.57}}, \frac{0.01}{n^{0.31}}\right), \; & \mbox {if}\; 1 \leq n \leq 2000,\\
&\left(\frac{0.002}{n^{0.51}}, \frac{0.002}{n^{0.69}}, \frac{0.01}{n^{0.8}}\right), \; & \mbox {otherwise}.
\end{aligned}
\right.\\
& \alpha_\phi  = \left(\frac{0.006}{n^{0.1}}, \frac{0.002}{n^{0.15}}\right).
\end{align*}}
The results for the learning of parameters are reported in Figure \ref{figure:Alg2B-NLQ}.

\begin{figure}[t!]
\centering
\begin{subfigure}{0.4\textwidth}
\centering
\includegraphics[width=\textwidth]{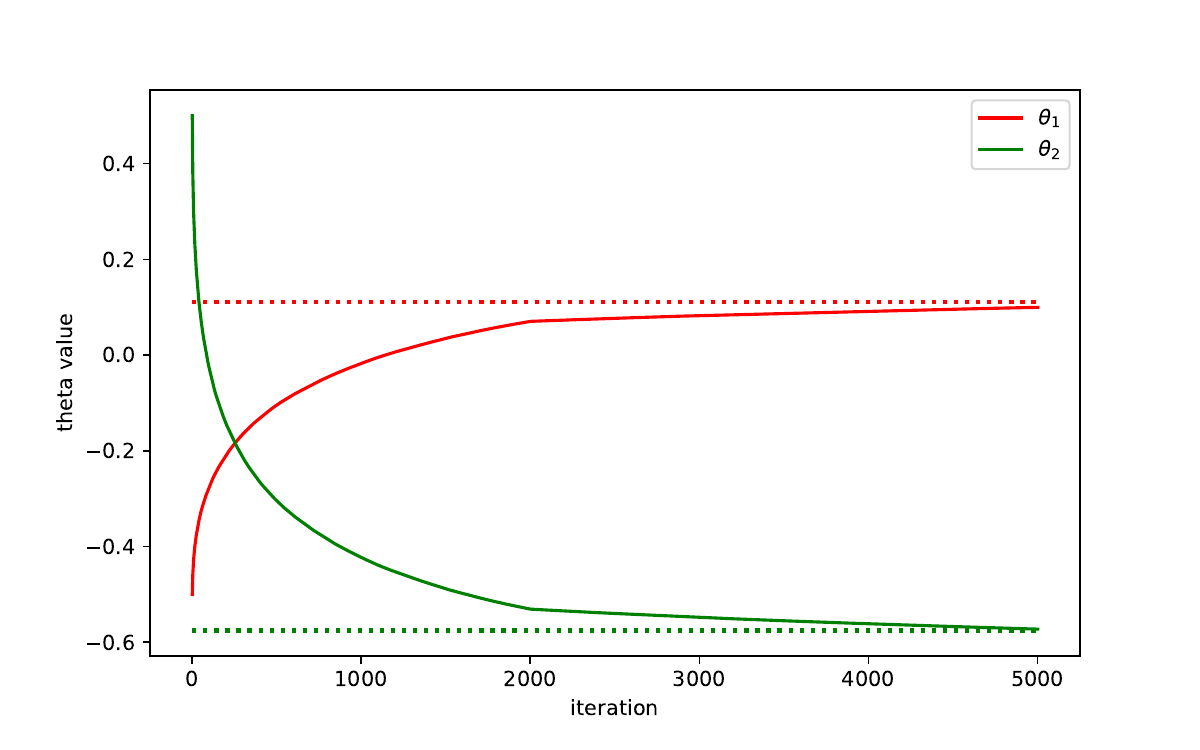}
\caption{Convergence of parameters $\theta$}
\label{NLQfigure:Alg2A-theta}
\end{subfigure}%
\begin{subfigure}{0.4\textwidth}
\centering
\includegraphics[width=\textwidth]{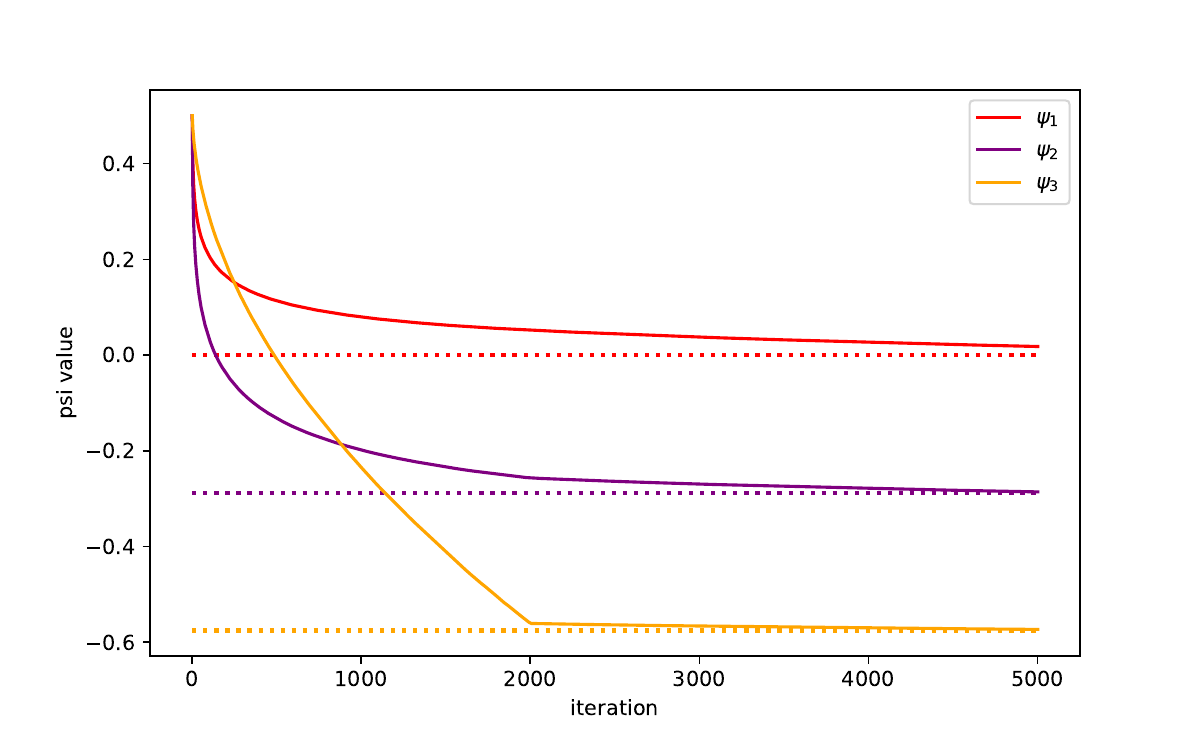}
\caption{Convergence of parameters $\psi$}
\label{NLQfigure:Alg2A-psi}
\end{subfigure}
\begin{subfigure}{0.4\textwidth}
\centering
\includegraphics[width=\textwidth]{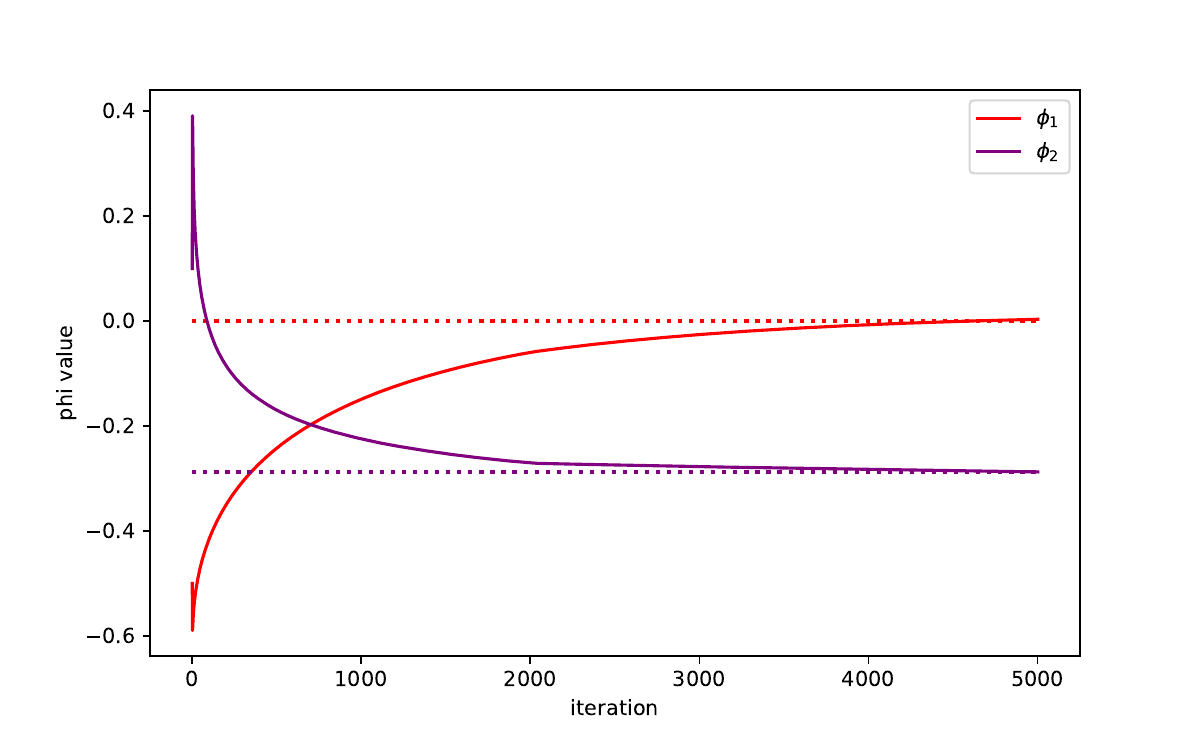}
\caption{Convergence of parameters $\phi$}
\label{NLQfigure:Alg2A-value}
\end{subfigure}
\begin{subfigure}{0.4\textwidth}
\centering
\includegraphics[width=\textwidth]{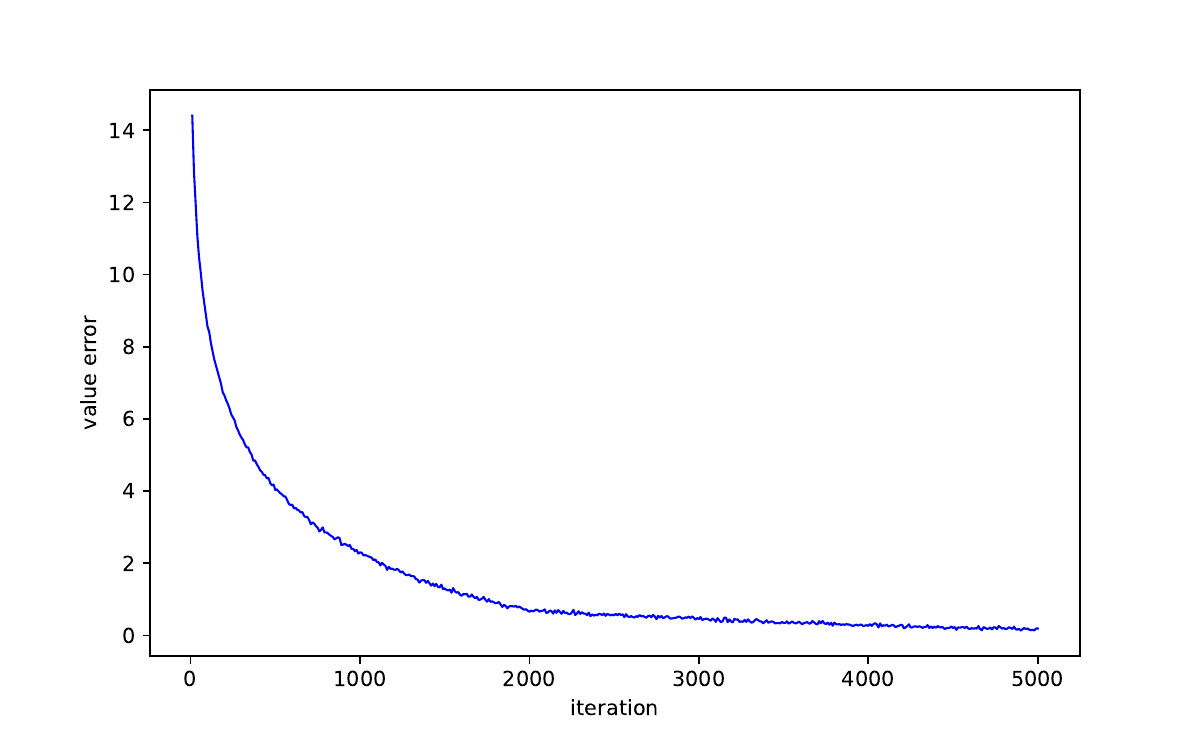}
\caption{$L^1$ error of the value function}
\label{NLQfigure:Alg2A-J}
\end{subfigure}
\caption{Convergence of the learnt parameters under Algorithm \ref{algo:actor-critic}-B }\label{figure:Alg2B-NLQ}
\end{figure}

We conclude with some interesting observations from two examples: (i)  All algorithms lead to the convergence of parameters  with the high accuracy within 10000 iterations. (ii) Although the outer iterations of Algorithm \ref{algo:actor-critic}-B is less than that of Algorithm \ref{algo:actor-critic}-A, the total iterations of Algorithm \ref{algo:actor-critic}-B is actually more than that of Algorithm \ref{algo:actor-critic}-A. Therefore, Algorithm \ref{algo:actor-critic}-A is more efficient and robust in cases when the characterization of the two-layer fixed point is not available and is also applicable to other cases such as when the normalizing constant in the Gibbs measure is not available (see section 5 in \cite{jiazhou2022}).

\vspace{0.2in}
\noindent
\textbf{Acknowledgement}:
X. Wei is supported by National Natural Science Foundation of China grant under no.12201343 and no.12571509.  X. Yu is supported by the Hong Kong RGC General Research Fund (GRF) under grant  no. 15211524.

\end{document}